\documentclass[letterpaper,12pt,reqno]{amsart}
\usepackage{amscd}
\usepackage{indentfirst}
\usepackage{graphics}
\numberwithin{equation}{section}
\usepackage[margin=2.9cm]{geometry}
\usepackage{epstopdf}
\usepackage{adjustbox}
\usepackage{quiver}

\usepackage{amsmath, amsthm, amsfonts, amssymb}
\usepackage{relsize}
\usepackage{bbm}
\usepackage{mathtools}
\usepackage{stmaryrd}
\usepackage{graphicx}
\usepackage{tikz-cd}
\usepackage{enumitem}

\definecolor{cof}{RGB}{219,144,71}
\definecolor{pur}{RGB}{186,146,162}
\definecolor{greeo}{RGB}{91,173,69}
\definecolor{greet}{RGB}{52,111,72}

\usepackage{enumitem}
\usepackage{hyperref}
\hypersetup{
    colorlinks=true,
    linkcolor=violet,
    filecolor=magenta,      
    urlcolor=blue,
    citecolor=red
}
\usetikzlibrary{decorations.markings,positioning}
\definecolor{pur}{RGB}{186,146,162}
\pgfdeclarelayer{background}
\pgfsetlayers{background,main}
\newtheorem{thm}{Theorem}[section]
\newtheorem*{thm*}{Theorem}
\newtheorem{prop}[thm]{Proposition}
\newtheorem{lem}[thm]{Lemma}
\newtheorem{cor}[thm]{Corollary}

\theoremstyle{definition}
\newtheorem{warn}[thm]{Warning}
\newtheorem{defn}[thm]{Definition}
\newtheorem*{defn*}{Definition}
\newtheorem{ex}[thm]{Example}
\newtheorem{rmk}[thm]{Remark}

\newcommand{\longhookrightarrow}{\lhook\joinrel\longrightarrow}

\newcommand{\tri}{\ol{\nabla}}

\newcommand{\bV}{\mathbb{V}}
\newcommand{\boxprod}{\hspace{0.2mm}\Box\hspace{0.2mm}}
\newcommand{\lifts}{\boxslash}

\newcommand{\cin}{\subseteq}

\newcommand\ol[1]{\ensuremath{\overline{#1}}}
\newcommand\abs[1]{\ensuremath{\lvert#1\rvert}}

\newcommand{\nospace}[1]{\makebox[0pt][l]{\,#1}}

\newcommand{\mcT}{\mathcal{T}}

\newcommand{\mcS}{\mathcal{S}}

\newcommand{\mcC}{\mathcal{C}}

\newcommand{\mcB}{\mathcal{B}}

\newcommand{\mcM}{\mathcal{M}}

\newcommand{\op}{\operatorname{op}}
\newcommand{\id}{\operatorname{id}}

\newcommand{\Hom}{\mathsf{Hom}}

\newcommand{\core}{\operatorname{core}}
\newcommand{\esd}{\operatorname{esd}}

\newcommand{\InnHorn}{\mathsf{IH}}
\newcommand{\GenInnHorn}{\mathsf{GIH}}
\newcommand{\TwoSegHorn}{\mathsf{2SH}}
\newcommand{\IsoHorn}{\mathsf{IsoHorn}}
\newcommand{\Bdry}{\mathsf{Bdry}}
\newcommand{\Mono}{\mathsf{Mono}}

\newcommand{\Set}{\mathsf{Set}}
\newcommand{\sSet}{\mathsf{sSet}}
\newcommand{\ssSet}{\mathsf{ssSet}}
\newcommand{\Cat}{\mathsf{Cat}}
\newcommand{\Gpd}{\mathsf{Gpd}}
\newcommand{\Ab}{\mathsf{Ab}}
\newcommand{\Sp}{\mathsf{Sp}}

\newcommand{\Map}{\mathsf{Map}}

\newcommand{\arxiv}[1]{\href{http://arxiv.org/pdf/#1}{arXiv:#1}}

\begin{document}

\title{Quasi-2-Segal sets}

\author[M. Feller]{Matt Feller}

\address{Max Planck Institute for Mathematics, 53111 Bonn, Germany}

\email{feller@virginia.edu}

\date{\today}

\thanks{The author was partially supported by NSF RTG grant DMS-1839968 and NSF grant DMS-1906281.}

\begin{abstract}
    We show that the 2-Segal spaces (also called decomposition spaces) of Dyckerhoff-Kapranov and G\'alvez-Kock-Tonks have a natural analogue within simplicial sets, which we call quasi-2-Segal sets, and that the two ideas enjoy a similar relationship as the one Segal spaces have with quasi-categories. In particular, we construct a model structure on the category of simplicial sets whose fibrant objects are the quasi-2-Segal sets which is Quillen equivalent to a model structure for complete 2-Segal spaces (where our notion of completeness comes from one of the equivalent characterizations of completeness for Segal spaces). We also prove a path space criterion, which says that a simplicial set is a quasi-2-Segal set if and only if its path spaces (also called d\'ecalage) are quasi-categories, as well as an edgewise subdivision criterion.
\end{abstract}

\maketitle

\setcounter{tocdepth}{1}

\tableofcontents

\section{Introduction}

In the decade since its introduction, the theory of 2-Segal spaces has seen a variety of applications and connections to a number of diverse areas of mathematics. 2-Segal spaces were first studied by Dyckerhoff-Kapranov in 2012 (in a preprint version of \cite{DK}), with applications ranging from representation theory and Hall algebras to homological algebra as well as geometry. Around the same time,  G{\'a}lvez-Kock-Tonks \cite{GKT:partI}, \cite{GKT:partII} independently began exploring the equivalent notion of decomposition spaces in order to extend the theory of M\"obius inversion beyond the realm of posets and categories. An important example of a 2-Segal space which both sets of authors identified independently is the output of Waldhausen's $S_{\bullet}$ construction from algebraic $K$-theory. Subsequent work by other authors has further expanded the scope of the theory; in \cite{BOORS:objects} it is shown that 2-Segal objects can be viewed equivalently through the lens of double categories, in \cite{Stern} they are seen to be equivalent to algebras of spans, while \cite{Walde:operads} shows that they can also be viewed as invertible operads. There has also been further development in the direction of Hall algebras in \cite{Young}, as well as continued work exploring M\"obius inversion from G{\'a}lvez-Kock-Tonks and others, such as \cite{Carlier}.

The idea at the heart of 2-Segal theory is that of a 2-Segal set, a simplicial set which behaves like a ``multi-valued category,'' which has objects, morphisms, and notion of composition which is associative in a certain sense, but which is not necessarily unique or even defined for a given pair of morphisms $x\to y$ and $y\to z$. A 2-Segal space is a simplicial space which behaves like a multi-valued category up to homotopy, in the same way that Segal spaces are simplicial spaces that behave like categories up to homotopy. The reason for the ``2'' in ``2-Segal'' becomes more clear from the explicit definition, which we give in Section \ref{sec:background}. The idea is that the associativity encoded by the 2-Segal condition can be expressed as a kind of composition of 2-simplices, compared to the composition of 1-simplices encoded by the Segal condition.

Much unlike the situation for Segal spaces and categories, there are relatively few examples of (strict) 2-Segal sets compared to the homotopical version, 2-Segal spaces. To illustrate why examples naturally tend to be up-to-homotopy, it is helpful to consider the $S_{\bullet}$ construction mentioned above, as it actually characterizes 2-Segal spaces in a certain sense as shown in \cite{BOORS:objects}. Let us stick with a relatively friendly input, the category of abelian groups $\Ab$, to show that homotopy considerations enter even in the most basic case. The simplicial object $S_{\bullet}(\Ab)$ has as $n$-simplices diagrams of abelian groups of the form




\[
\adjustbox{scale=0.8}{
\begin{tikzcd}[row sep=small, column sep=small]
0 \arrow[r, hook] & {A_{1,0}} \arrow[d, two heads] &  & 0 \arrow[r, hook] & {A_{1,0}} \arrow[d, two heads] \arrow[r, hook] & {A_{2,0}} \arrow[d, two heads] &  & 0 \arrow[r, hook] & {A_{1,0}} \arrow[r, hook] \arrow[d, two heads] & {A_{2,0}} \arrow[r, hook] \arrow[d, two heads] & {A_{3,0}} \arrow[d, two heads] \\
                  & 0                              &  &                   & 0 \arrow[r, hook]                              & {A_{2,1}} \arrow[d, two heads] &  &                   & 0 \arrow[r, hook]                              & {A_{2,1}} \arrow[d, two heads] \arrow[r, hook] & {A_{3,1}} \arrow[d, two heads] \\
                  &                                &  &                   &                                                & 0                              &  &                   &                                                & 0 \arrow[r, hook]                              & {A_{3,2}} \arrow[d, two heads] \\
                  &                                &  &                   &                                                &                                &  &                   &                                                &                                                & 0                             
\end{tikzcd}
}
\]
for $n=1,2,3$, where the horizontal maps are injective, the vertical maps are surjective, and each square is both a pushout and a pullback. Each face map $d_i$ amounts to deleting the entries $A_{j,k}$ such that $j$ or $k$ equals $i$. In $S_{\bullet}(\Ab)$, the 1-simplices are effectively abelian groups (since the maps to and from $0$ add no data), and 
the 2-simplices are short exact sequences. That is, given two abelian groups $A$ and $B$, a choice of composite of $A$ and $B$ is a short exact sequence
\[
\begin{tikzcd}[column sep=small]
0 \arrow[r, hook] & {A} \arrow[r, hook] & {C} \arrow[r, two heads] & {B} \arrow[r, two heads] & 0\nospace{.}
\end{tikzcd}
\]

In the context of $S_{\bullet}(\Ab)$, the 2-Segal condition amounts to saying that a 3-simplex diagram, as shown above, is uniquely determined by the sub-diagram with the $A_{3,1}$ term missing, or by the sub-diagram with the $A_{2,0}$ term missing. Where the ``up to homotopy'' consideration enters here is that pullbacks and pushouts are only determined up to isomorphism, meaning that if we were to define $S_{\bullet}(\Ab)$ as a simplicial set, it would not quite satisfy the strict 2-Segal condition. The usual way to address this issue is to define $S_{\bullet}(\Ab)$ instead as a simplicial groupoid $\Delta^{\op}\to \Gpd$ or as a simplicial space $\Delta^{\op}\to \sSet$, making $S_{\bullet}(\Ab)$ a 2-Segal space.

When considering up-to-homotopy versions of categories, Segal spaces are just one of the different models in the literature. Some work has been done to develop Segal space theory---for example, in \cite{Rezk:CSS} and \cite{Rasekh:Yoneda}---but the model with the most robust theory is that of quasi-categories, thanks largely to Joyal \cite{Joyal:theory} and Lurie \cite{Lurie:HTT}. The two ideas are equivalent via two adjunctions between simplicial sets $\sSet$ and simplicial spaces $\ssSet$, as shown in \cite{JT}. These adjunctions give us a way to move between the two models, exploiting the nice properties of one model or the other depending on the situation.

Recall that a quasi-category is a simplicial set $X$ with fillers of all horn inclusions $\Lambda^i[n]\hookrightarrow \Delta[n]$ which are inner (meaning $0<i<n$), where the horn $\Lambda^i[n]$ is the union of all but the $d_i$ face of $\Delta[n]$. We think of the 0-simplices of $X$ as the objects, the 1-simplices of $X$ as the morphisms, and the 2-simplices as witnessing its $d_1$ face as a composite of its $d_2$ and $d_0$ faces. Having fillers with respect to $\Lambda^1[2]\hookrightarrow \Delta[2]$ tells us that there always exists a composite of two morphisms $x\to y$ and $y\to z$, but does not say anything about uniqueness. Instead, the horn conditions for $n\geq 3$ end up implying that any choice of composition is unique up to coherent homotopy.

Let us denote by $S^{\operatorname{set}}_{\bullet}(\Ab)$ the version of $S_{\bullet}(\Ab)$ where we only take sets instead of groupoids or spaces. Given the data of a 3-simplex diagram except for the $A_{3,1}$ term, even though there are many choices of pushout which fill in the diagram, they are all isomorphic. We therefore might ask whether this isomorphism is witnessed by the higher simplices of $S^{\operatorname{set}}_{\bullet}(\Ab)$, just like how in a quasi-category different composites are homotopic as witnessed by higher simplices. In other words, is $S^{\operatorname{set}}_{\bullet}(\Ab)\colon \Delta^{\op}\to \Set$ an example of what one might call a \emph{quasi-2-Segal set}? The goal of this paper is to answer ``yes!''

To give our definition of quasi-2-Segal sets, we must describe a 2-Segal analogue of inner horns. Instead of an ordinary horn $\Lambda^i[n]$ which is the union of all but one face, we take horns of the form $\Lambda^{i,j}[n]$ where we take the union of all but the $d_i$ and $d_j$ faces. Instead of the inner condition, we require that $i< j-1$ and that if $i=0$ then $j\neq n$; in other words, $i$ and $j$ are not adjacent mod $(n+1)$. We call a horn $\Lambda^{i,j}[n]$ satisfying this condition a \emph{2-Segal horn}. We can then define quasi-2-Segal sets in analogy with quasi-categories.

\vspace{2mm}

\noindent\textbf{Definition \ref{def:q2seg}.} A simplicial set $X$ is a \emph{quasi-2-Segal set} if it has fillers of all 2-Segal horn inclusions $\Lambda^{i,j}[n]\hookrightarrow \Delta[n]$.

\vspace{2mm}

One can check that $S^{\operatorname{set}}_{\bullet}(\Ab)$ satisfies this definition, using arguments similar to those sketched above for the $n=3$ case.

Just as in the case of quasi-categories, it is not immediately clear from the definition that higher simplices provide a good notion of homotopy in quasi-2-Segal sets. Hence, a primary goal of this paper is to prove the following 2-Segal analogues of fundamental results from quasi-category theory.

\begin{thm*}
The following hold:
\begin{enumerate}
    \item \emph{(Corollary~\ref{cor:contractiblespace})} A simplicial set $X$ is quasi-2-Segal if and only if the induced maps $\Map(\Delta[3],X)\to \Map(\Lambda^{0,2}[3],X)$ and $\Map(\Delta[3],X)\to \Map(\Lambda^{1,3}[3],X)$ are trivial fibrations.
    \item \emph{(Lemma~\ref{lem:invertinghomotopy}, Proposition~\ref{prop:augmentedhornlifting})} The $n$-simplices in a quasi-2-Segal set $X$ with invertible edge $i\to i+1$ behave like ``homotopies of $(n-1)$-simplices.''
    \item \emph{(Theorem~\ref{thm:modelstructure})} There exists a model structure on the category of simplicial sets whose fibrant objects are precisely the quasi-2-Segal sets.
\end{enumerate}
\end{thm*}

In particular, we see that for $S^{\operatorname{set}}_{\bullet}(\Ab)$, the choices of fillers of horns of the form $\Lambda^{0,2}[3]$ and $\Lambda^{1,3}[3]$ are indeed unique up to homotopy in the sense we would hope for.

An important result from 2-Segal theory, proved independently in \cite[Thm.~6.3.2]{DK} and \cite[Thm.~4.10]{GKT:partI}, is the path space criterion, which says that a simplicial space is 2-Segal if and only if its path spaces are Segal. We prove the analogous criterion for quasi-2-Segal sets.

\vspace{2mm}

\noindent\textbf{Theorem \ref{thm:pathspacecriterion}} (Path space criterion)\textbf{.} \emph{A simplicial set is a quasi-2-Segal set if and only if its path spaces are quasi-categories.}

\vspace{2mm}

In particular, this criterion tells us that one only needs to check for fillers of horns of the form $\Lambda^{0,j}[n]$ for $2\leq j\leq n$ and $\Lambda^{i,n}[n]$ for $0\leq i\leq n-2$ to see if a simplicial set is a quasi-2-Segal set. Together with the path space criterion from 2-Segal spaces, it also unlocks a way to get examples of quasi-2-Segal sets as discrete versions of 2-Segal spaces, similar to how $S^{\operatorname{set}}_{\bullet}(\Ab)$ is the discrete version of $S_{\bullet}(\Ab)$.

Our final main result recreates the equivalence between quasi-categories and complete Segal spaces proved in \cite{JT}. To do so, we first define a notion of completeness for 2-Segal spaces. Recall that a Segal space $W$ is complete if a certain subspace $W_{hoequiv}\cin W_1$ of homotopy equivalences is weakly equivalent to $W_0$. However, for Segal spaces, this condition turns out to be equivalent to being local with respect to a certain inclusion $p_1^{\ast}(\{0\}\hookrightarrow J)$. We choose this latter condition as our definition for completeness of 2-Segal spaces. By localizing the model structure for 2-Segal spaces at $p_1^{\ast}(\{0\}\hookrightarrow J)$, we get a model structure on simplicial spaces where the fibrant objects are precisely the complete 2-Segal spaces, which we show is Quillen equivalent to the model structure for quasi-2-Segal sets.

\vspace{2mm}

\noindent\textbf{Theorem \ref{thm:quillenequiv}.} \emph{The quasi-2-Segal set model structure is Quillen equivalent to the model structure for complete 2-Segal spaces via the same adjunctions as for quasi-categories and complete Segal spaces.}

\vspace{2mm}

In the same way that the Quillen equivalences from \cite{JT} witness complete Segal spaces and quasi-categories as equivalent models of $(\infty,1)$-categories, we can interpret our result as saying that complete 2-Segal spaces and quasi-2-Segal sets are equivalent models of $(\infty,1)$-multi-valued categories. We hope that further work in 2-Segal theory might benefit from the interplay between the two models.

Our work has just scratched the surface in terms of generalizing the vast theory of quasi-categories. As just one example of possible future work, it would be nice to have a 2-Segal version of Quillen's Theorem A, generalizing the one for quasi-categories \cite[Thm.~4.1.3.1]{Lurie:HTT}. But we also expect that the path space criterion and edgewise subdivision criterion, which say that a simplicial set is quasi-2-Segal if and only if certain constructions are quasi-categories, will allow for a direct application of quasi-category theory in some situations. One avenue for future work, attempting to generalize our results to the $d$-Segal case for $d>2$, turns out to be somewhat of a dead end for reasons we explain in Section \ref{sec:highersegal}.

\subsection{Organization}

In Section \ref{sec:background}, we give a broad overview of the necessary background for the paper. In Section \ref{sec:2segalhorns}, we define 2-Segal horns and quasi-2-Segal sets and show that one arrives naturally at our definition of 2-Segal horns by beginning the ``2-Segal spines'' which define 2-Segal spaces. In Section \ref{sec:pathspacecriterion}, we prove for quasi-2-Segal sets a version of the path space criterion for 2-Segal spaces. In Section \ref{sec:pushoutprod}, we prove some technical results about pushout-products and pushout-joins of 2-Segal horns which yield the characterization of quasi-2-Segal sets in terms of having a contractible space of fillers of horns of the form $\Lambda^{0,2}[3]$ and $\Lambda^{1,3}[3]$. In Section \ref{sec:modelstructure}, we show that quasi-2-Segal sets satisfy a version of the special horn lifting property from quasi-categories, and we show that there is a model structure for quasi-2-Segal sets. In Section \ref{sec:vs} we define a notion of completeness for 2-Segal spaces, and we show that the complete 2-Segal space model structure is Quillen equivalent to the model structure for quasi-2-Segal sets. We conclude with Section \ref{sec:highersegal}, which is a brief remark about the issues with extending this theory to $d$-Segal objects for $d>2$.

\subsection{Acknowledgements}

I would like to thank Julie Bergner for suggesting this topic and for her valuable feedback on the write-up, Tim Campion for pointing out that left and right inverses of 1-simplices often do not agree in $S_{\bullet}(\Sp)$, Nima Rasekh for his comments on an earlier draft, and Walker Stern for a number of helpful conversations and for teaching a wonderful course on quasi-categories which inspired some of the results proved here. I am also grateful to the referee for their incisive feedback and thoughtful suggestions.

\section{Background}\label{sec:background}

\subsection{Simplicial sets} We give a brief overview of our terminology and notation for simplicial sets; see \cite[\S 1.2]{Cisinski:Cambridge} or \cite{Rezk} for a more thorough treatment.

The \emph{simplex category $\Delta$} is the category with objects the finite linearly ordered sets
\[
[n]=\{0\leq 1\leq \ldots\leq n\}\text{ for each }n\geq 0,
\]
whose morphisms are monotone maps. A \emph{simplicial set} is a functor $\Delta^{\op}\to \Set$. We denote the category of simplicial sets by $\sSet$. We denote by $\Delta[n]$ the standard $n$-simplex, which is the simplicial set corresponding to $[n]$ via the Yoneda embedding $\Delta\hookrightarrow \sSet$.

Given a simplicial set $X\colon \Delta^{\op}\to \Set$, we denote by $X_n$ the set to which $X$ maps $[n]$, and we call an element in $X_n$ an \emph{$n$-simplex} of $X$. We often refer to a 0-simplex as a \emph{vertex} and a 1-simplex as an \emph{edge}. By the Yoneda Lemma, an $n$-simplex in $X$ is equivalently a map of simplicial sets $\Delta[n]\to X$.

The injective morphisms in $\Delta$ are generated by co-face maps $d^i\colon [n-1]\hookrightarrow [n]$, which give us the \emph{face maps} $d_i\colon X_n\to X_{n-1}$ of a simplicial set $X$ for each $n\geq 1$ and $0\leq i\leq n$. The \emph{boundary} of the standard $n$-simplex $\Delta$ is the union of all of its faces, denoted by $\partial\Delta[n]$.

The surjective morphisms in $\Delta$ are generated by co-degeneracy maps $s^i\colon [n+1]\to [n]$, which give us the \emph{degeneracy maps} $s_i\colon X_n\to X_{n+1}$ of a simplicial set $X$ for each $n\geq 0$ and $0\leq i\leq n$. An $n$-simplex is \emph{degenerate} if it is in the image of a degeneracy map, and is \emph{non-degenerate} otherwise.

We say that a simplicial set $Y$ is a \emph{subcomplex} of $X$ if $Y_n\cin X_n$, and these inclusions form a simplicial map $X\to Y$. Given a simplicial set $X$ and a set of vertices $\nu\cin X$, the \emph{full subcomplex of $X$ on $\nu$} is the subcomplex $Z\cin X$ consisting of the simplices of $X$ whose vertices are all in $\nu$.

Given two simplicial sets $X$ and $Y$, there is a simplicial set $\Map(X,Y)$ where
\[
(\Map(X,Y))_n=\Hom(X\times \Delta[n], Y).
\]

\subsection{The nerve functor}

Denote by $\Cat$ the category of small categories. The \emph{nerve functor} is a fully faithful embedding $N\colon \Cat\hookrightarrow \sSet$; see \cite[\S 1.4]{Cisinski:Cambridge}. For example, the nerve of the category $0\to 1\to \ldots \to n$ is the standard $n$-simplex $\Delta[n]$. The nerve functor has a left adjoint which we denote by $\tau_1$; see \cite{Rezk} for an explicit description of this functor (the ``fundamental category'' functor, there denoted by $h$).

Denote by $\mathbb{I}$ the free-living isomorphism, i.e., the category with two objects and exactly one morphism in each hom-set. We let $J=N(\mathbb{I})$, and denote its two vertices by 0 and 1.

\subsection{Lifting properties and saturated classes}\label{sub:lifts}

We say that a map $g\colon X\to Y$ has the \emph{right lifting property with respect to} $f\colon A\to B$, denoted by $f\lifts g$, if for all commutative squares
\[
\begin{tikzcd}
{A} \arrow[d, "f"'] \arrow[r, "u"]                      & {X} \arrow[d, "g"] \\
{B} \arrow[r, "v"'] \arrow[ru, "\exists \ell"', dotted] & {Y} \nospace{,}             
\end{tikzcd}
\]
there exists a \emph{lift}, i.e., a dotted arrow $\ell$ making each triangle commute. More generally, we say that a map $g$ has the \emph{right lifting property with respect to} a class of maps $\mathcal{B}$
if $f\lifts g$ for all $f$ in $\mathcal{B}$.


A class of morphisms is \emph{saturated} if it is closed under taking isomorphisms, pushouts, transfinite compositions, and retracts. Given a class of morphisms $\mcB$, we can take its \emph{saturated closure} $\ol{\mcB}$. If a class $\mcB'$ equals $\ol{\mcB}$, we say that $\mcB'$ is \emph{generated by $\mcB$}. For example, the set of \emph{boundary inclusions} $\Bdry=\{\partial\Delta[n]\to\Delta[n]\}_{n\geq 0}$ generates the class $\Mono$ of monomorphisms in $\sSet$. A map has the right lifting property with respect to a set $\mcS$ if and only if it has the right lifting property with respect to the class $\ol{\mcS}$.
For any classes of maps $\mcS$ and $\mcT$, the containment $\mcS\cin \ol{\mcT}$ implies $\ol{\mcS}\cin \ol{\mcT}$. See \cite{Rezk} or \cite[\S 2.1]{Cisinski:Cambridge} for more details about lifting properties and saturated classes.

\subsection{Model structures}\label{sub:modelstructures}

A \emph{model structure} on a complete and cocomplete category consists of a choice of three classes of morphisms, the \emph{cofibrations}, the \emph{fibrations}, and the \emph{weak equivalences}, subject to certain axioms; see \cite[Def.~2.2.1]{Cisinski:Cambridge} or \cite{Hovey}. We say that a morphism which is both a (co)fibration and a weak equivalence is a \emph{trivial (co)fibration}. The \emph{fibrant objects} are those such that the map to the terminal object is a fibration, and the \emph{cofibrant objects} are those such that the map from the initial object is a cofibration.

Instead of describing the axioms of a model category in general, let us only state the main properties we use in this paper.

\begin{enumerate}
    \item The class of weak equivalences satisfies the 2-out-of-3 property: if two of $f$, $g$, and $gf$ are weak equivalence, then so is the third.
    \item The class of trivial cofibrations forms a saturated class.
    \item \cite[Prop.~E.1.10]{Joyal:theory} A model structure is determined by its class of cofibrations together with its class of fibrant objects.
\end{enumerate}

We restrict our focus to \emph{Cisinski model structures}, which are cofibrantly generated model structures on a presheaf category whose cofibrations are precisely the monomorphisms. A model structure is \emph{cofibrantly generated} if there are sets $\mathcal{I}$ and $\mathcal{J}$ such that $\mathcal{I}$ generates the cofibrations and $\mathcal{J}$ generates the trivial cofibrations in the sense of Subsection \ref{sub:lifts}; see \cite[2.4.1]{Cisinski:Cambridge}. Given two Cisinski model structures $\mathcal{M}$ and $\mathcal{M}'$ on $\sSet$ whose classes of weak equivalences are $W$ and $W'$ respectively, we say that $\mathcal{M}'$ is a \emph{localization of $\mathcal{M}$} if $W'\supseteq W$. Given a Cisinski model structure $\mathcal{M}$ and a set of morphisms $S$, the \emph{localization of $\mathcal{M}$ at $S$} is the model structure with the smallest class of weak equivalences which contains $S$ as well as the weak equivalences of $\mathcal{M}$. By fact (3) above, a Cisinski model structure is uniquely determined by its class of fibrant objects.

\subsection{Horns and Kan complexes}

Given $n\geq 1$ and $0\leq i\leq n$, the union of every face of $\Delta[n]$ except $d_i$ is a \emph{horn}, denoted by $\Lambda^i[n]$. More generally, given some proper subset $S\cin \{0,1,\ldots,n\}$, a \emph{generalized horn} $\Lambda^{S}[n]$ is the union of every face $d_i\Delta[n]$ for $i\in S$. Note that this notation is ``additive,'' in the sense that the $S$ in $\Lambda^S[n]$ tells us which faces are included instead of excluded. This notation is sometimes convenient, but it is also often convenient to denote by $\Lambda^{i_1,i_2,\ldots,i_r}[n]$ the generalized horn missing the faces $d_{i_1},\ldots, d_{i_r}$.

A simplicial set $X$ is a \emph{Kan complex} if the map $X\to \ast$ has the right lifting property with respect to the set of all horn inclusions. We think of a Kan complex as a simplicial set which behaves like a topological space. A way to make this statement more precise is via the Kan-Quillen model structure on $\sSet$ \cite{Quillen}, which is a Cisinski model structure whose fibrant objects are the Kan complexes. One can show that this model structure is equivalent to a model structure on the category of topological spaces.

\subsection{Inner horns and quasi-categories}

We say that a face $d_i\Delta[n]$ is \emph{inner} if $0<i<n$ and \emph{outer} if $i=0$ or $i=n$. Similarly, a horn $\Lambda^i[n]$ is \emph{inner} if $0<i<n$ and outer if $i=0$ or $i=n$. If, for a set $S\cin \{0,1,\ldots,n\}$, there exist $i<j<k$ such that $i,k\in S$ and $j\not\in S$, then the generalized horn $\Lambda^{S}[n]$ is a \emph{generalized inner horn}. We denote the set of all inner horn inclusions by $\InnHorn$, and the set of all generalized inner horn inclusions by $\GenInnHorn$. One can show that $\ol{\InnHorn}=\ol{\GenInnHorn}$; see an appendix in \cite{Rezk} for a proof.

A simplicial set $X$ is a \emph{quasi-category} if $X\to \ast$ has the right lifting property with respect to $\InnHorn$. We think of the 0-simplices of a quasi-category as objects and the 1-simplices as morphisms. There is a Cisinski model structure on $\sSet$ due to Joyal \cite{Joyal:theory} whose fibrant objects are the quasi-categories.

Every quasi-category has the \emph{special outer horn lifting property} as shown by Joyal \cite{Joyal:published}, which says that a quasi-category $X$ does have fillers of outer horns $\Lambda^0[n]\to X$ (resp.~$\Lambda^n[n]$) as long as the $0\to 1$ edge (resp.~$(n-1)\to n$ edge) is sent to an invertible morphism.

\subsection{Simplicial spaces}

A \emph{simplicial space} is a functor $W\colon \Delta^{\op}\to \sSet$. We denote by $\ssSet$ the category of simplicial spaces. Since each $W_n$ for $n\geq 0$ is a simplicial set, we can view a simplicial space as a grid of sets, where the $n$th column is the simplicial set $W_n$. There is a functor $p_1^{\ast}\colon \sSet\to \ssSet$ where $(p_1^{\ast}(X))_n$ is the discrete simplicial set corresponding to the set $X_n$; in other words, $p_1^{\ast}$ gives us a horizontal embedding of $\sSet$ into $\ssSet$. Similarly, there is a vertical embedding $p_2^{\ast}$, and given two simplicial spaces $W$ and $Z$, there is a simplicial set $\Map(W,Z)$ given by $(\Map(W,Z))_n =\Hom(W\times p_2^{\ast}(\Delta[n]),Z)_n$. In particular, $\Map(p_1^{\ast}(\Delta[n]),W)\cong W_n$ for every $n\geq 0$.

There is a model structure on $\ssSet$ which we call the \emph{vertical Reedy model structure with respect to the Kan-Quillen model structure}, or just the \emph{vertical Reedy model structure}, whose fibrant objects are those such that for every $n\geq 0$, the map
\[\Map(p_1^{\ast}(\Delta[n]),W)\to \Map(p_1^{\ast}(\partial\Delta[n]),W)\] induced by the boundary inclusion is a Kan-Quillen fibration. If $W$ is a vertical Reedy fibrant simplicial space, then in particular $W_n$ is a Kan complex for every $n\geq 0$. We use the word ``vertical'' here because we follow the convention of \cite{JT} and picture the simplicial sets $W_n$ as the columns of the simplicial space $W$.

We say a simplicial space $W$ is \emph{local} with respect to a map $f\colon Z\to Z'$ if the induced map $\Map(Z',W)\to \Map(Z,W)$ is a Kan-Quillen weak equivalence. Given a set of maps $S$, the fibrant objects in the vertical Reedy fibrant model structure localized at $S$ is precisely the Reedy fibrant simplicial spaces which are local with respect to $S$. See \cite{Hirschhorn} for more about localizations of this sort.

\subsection{Segal objects and spines}

Given $n\geq 2$, we say that the union of the edges $i\to (i+1)$ in $\Delta[n]$ is its \emph{spine}, denoted by $Sp[n]$. The inclusion $Sp[n]\hookrightarrow \Delta[n]$ is a \emph{spine inclusion}. A simplicial set $X$ is isomorphic to the nerve of a category precisely if it has unique fillers of spine inclusions. A homotopical version of this idea yields the notion of a \emph{Segal space}, which is a vertical Reedy fibrant simplicial space $W$ such that each induced map $\Map(p_1^{\ast}(\Delta[n]),W)\to \Map(p_1^{\ast}(Sp[n]),W)$ is a Kan-Quillen weak equivalence. In other words, a Segal space is a fibrant object in the localization of vertical Reedy model structure at the horizontal embedding of the set of spine inclusions.

We say that a Segal space is \emph{complete} if it is local with respect to the map $p_1^{\ast}(\{0\}\hookrightarrow J)$. (Note that this definition is equivalent to, but not precisely the same as, the original given in \cite{Rezk:CSS}).) Complete Segal spaces give an alternative notion of $(\infty,1)$-category to that of quasi-categories, although the two are equivalent as shown in \cite{JT}.

\subsection{2-Segal objects and 2-Segal spines}\label{sub:2segspine}

We can define a two-dimensional analogue of spines, which we call \emph{2-Segal spines}. Given $n\geq 3$, a 2-Segal spine $\mcT\cin \Delta[n]$ is a union of 2-simplices of $\Delta[n]$ which gives a triangulation of the $(n+1)$-gon formed by the edges $0\to n$ and $i\to i+1$ for each $0\leq i<n$. For example, here are two 2-Segal spines inside of $\Delta[5]$:




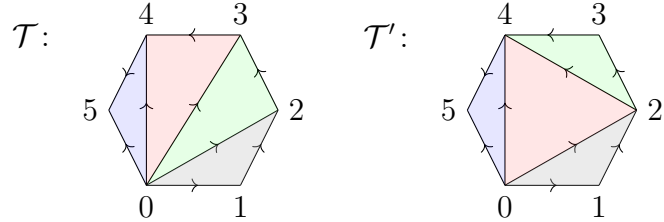
\[
\begin{tikzpicture}
\coordinate[label=below:$0$] (0) at (-1,0);
\coordinate[label=below:$1$] (1) at (0.25,0);
\coordinate[label=right:$2$] (2) at (0.75,1);
\coordinate[label=above:$3$] (3) at (0.25,2);
\coordinate[label=above:$4$] (4) at (-1,2);
\coordinate[label=left:$5$] (5) at (-1.5,1);
\coordinate[label=left:$\mcT\colon$] (l) at (-2,2);
\begin{scope}[decoration={markings,mark=at position 0.55 with {\arrow{>}}}] 
\draw[postaction={decorate}] (0) -- node[auto] {} (1);
\draw[postaction={decorate}] (1) -- node[auto] {} (2);
\draw[postaction={decorate}] (2) -- node[auto] {} (3);
\draw[postaction={decorate}] (3) -- node[auto] {} (4);
\draw[postaction={decorate}] (4) -- node[auto] {} (5);
\draw[postaction={decorate}] (0) -- node[auto] {} (2);
\draw[postaction={decorate}] (0) -- node[auto] {} (3);
\draw[postaction={decorate}] (0) -- node[auto] {} (4);
\draw[postaction={decorate}] (0) -- node[auto] {} (5);
\end{scope}
\path[fill=lightgray,fill opacity=.3] (0)--(1)--(2)--cycle;
\path[fill=green,fill opacity=.1] (0)--(2)--(3)--cycle;
\path[fill=red,fill opacity=.1] (0)--(3)--(4)--cycle;
\path[fill=blue,fill opacity=.1] (0)--(4)--(5)--cycle;
\end{tikzpicture}
\hspace{.2in}
\begin{tikzpicture}
\coordinate[label=below:$0$] (0) at (-1,0);
\coordinate[label=below:$1$] (1) at (0.25,0);
\coordinate[label=right:$2$] (2) at (0.75,1);
\coordinate[label=above:$3$] (3) at (0.25,2);
\coordinate[label=above:$4$] (4) at (-1,2);
\coordinate[label=left:$5$] (5) at (-1.5,1);
\coordinate[label=left:$\mcT'\colon$] (l) at (-2,2);
\begin{scope}[decoration={markings,mark=at position 0.55 with {\arrow{>}}}] 
\draw[postaction={decorate}] (0) -- node[auto] {} (1);
\draw[postaction={decorate}] (1) -- node[auto] {} (2);
\draw[postaction={decorate}] (2) -- node[auto] {} (3);
\draw[postaction={decorate}] (3) -- node[auto] {} (4);
\draw[postaction={decorate}] (4) -- node[auto] {} (5);
\draw[postaction={decorate}] (0) -- node[auto] {} (2);
\draw[postaction={decorate}] (2) -- node[auto] {} (4);
\draw[postaction={decorate}] (0) -- node[auto] {} (4);
\draw[postaction={decorate}] (0) -- node[auto] {} (5);
\end{scope}
\path[fill=lightgray,fill opacity=.3] (0)--(1)--(2)--cycle;
\path[fill=green,fill opacity=.1] (2)--(3)--(4)--cycle;
\path[fill=red,fill opacity=.1] (0)--(2)--(4)--cycle;
\path[fill=blue,fill opacity=.1] (0)--(4)--(5)--cycle;
\end{tikzpicture}.
\]
A \emph{2-Segal set} is a simplicial set with unique fillers of all 2-Segal spine inclusions. We can think of a 2-Segal set as being like a weak version of a category, where we view the 0-simplices as objects, the 1-simplices as morphisms, and the 2-simplices as witnessing composition of morphisms, except that the composite of two morphisms $x\to y\to z$ need not be unique or even defined. Having unique fillers of 2-Segal spine extensions implies that this weak notion of composition is still associative in a certain sense, given by the correspondence between the two 2-Segal spines of $\Delta[3]$:




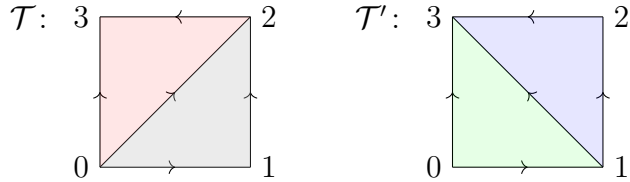
\[
\begin{tikzpicture}[node distance=2cm]
\coordinate[label=right:$2$] (c) {};
\coordinate[left=of c,label=left:$3$] (d) {};
\coordinate[left=of c, label=left:$\mcT\colon\quad$] (l) {};
\coordinate[below=of d,label=left:$0$] (a) {};
\coordinate[right=of a,label=right:$1$] (b) {};
\begin{scope}[decoration={markings,mark=at position 0.5 with {\arrow{>}}}] 
\draw[postaction={decorate}] (c) -- node[auto] {} (d);
\draw[postaction={decorate}] (a) -- node[auto] {} (d);
\draw[postaction={decorate}] (a) -- node[auto] {} (b);
\draw[postaction={decorate}] (b) -- node[auto] {} (c);
\draw[postaction={decorate}] (a) -- node[auto] {} (c);
\end{scope}
\path[fill=red,fill opacity=.1] (a)--(c)--(d)--cycle;
\path[fill=lightgray,fill opacity=.3] (a)--(c)--(b)--cycle;
\end{tikzpicture}
\hspace{.3in}
\begin{tikzpicture}[node distance=2cm]
\coordinate[label=right:$2$] (c) {};
\coordinate[left=of c,label=left:$3$] (d) {};
\coordinate[left=of c, label=left:$\mcT'\colon\quad$] (l) {};
\coordinate[below=of d,label=left:$0$] (a) {};
\coordinate[right=of a,label=right:$1$] (b) {};
\begin{scope}[decoration={markings,mark=at position 0.5 with {\arrow{>}}}] 
\draw[postaction={decorate}] (c) -- node[auto] {} (d);
\draw[postaction={decorate}] (a) -- node[auto] {} (d);
\draw[postaction={decorate}] (a) -- node[auto] {} (b);
\draw[postaction={decorate}] (b) -- node[auto] {} (c);
\draw[postaction={decorate}] (b) -- node[auto] {} (d);
\end{scope}
\path[fill=blue,fill opacity=.1] (b)--(c)--(d)--cycle;
\path[fill=green,fill opacity=.1] (a)--(b)--(d)--cycle;
\end{tikzpicture}.
\]

A \emph{2-Segal space}, first defined in \cite{DK}, is a vertical Reedy fibrant simplicial space which is local with respect to $p_1^{\ast}(\mcT\hookrightarrow \Delta[n])$ for every 2-Segal spine inclusion $\mcT\hookrightarrow \Delta[n]$. An equivalent notion of \emph{decomposition spaces} was defined in \cite{GKT:partI} (one part of a series of six papers), although the condition is stated very differently. In particular, it was only observed relatively recently in \cite{FGKPW} that a certain condition called unitality, which is built into the definition of decomposition spaces and was often assumed as an extra axiom for 2-Segal spaces, is actually automatic from the 2-Segal condition.

In general, Segal implies 2-Segal; every Segal space is a 2-Segal space and the nerve of a category (a ``Segal set'') is a 2-Segal set.

\subsection{Pushout-products and pushout-joins}

One can define pushout-products of arbitrary morphisms in a monoidal category, but for our purposes it suffices to consider monomorphisms in the monoidal category $(\sSet,\times)$. Given monomorphisms $A\hookrightarrow B$ and $C\hookrightarrow D$ of simplicial sets, the monomorphism
\[
(D\times A)\cup(C\times B)\hookrightarrow D\times B
\]
is called the \emph{pushout-product} of $A\hookrightarrow B$ and $C\hookrightarrow D$, denoted by $(A\hookrightarrow B)\boxprod (C\hookrightarrow D)$. Given two classes $\mcS$ and $\mcT$ of maps, we denote by $\mcS\boxprod\mcT$ the class of maps of the form $f\boxprod g$ for $f$ in $\mcS$ and $g$ in $\mcT$. Given morphisms $A\hookrightarrow B$, $C\hookrightarrow D$, and $X\to Y$, there is a correspondence of lifting problems




\[\begin{tikzcd}[column sep=small]
	{(A\times D)\cup (B\times C)} & X &&& C & {\Map(B,X)} \\
	{B\times D} & Y & {} & {} & D & {\Map(A,X)\times_{\Map(A,Y)}\Map(B,Y)\nospace{,}}
	\arrow[from=2-1, to=2-2]
	\arrow[from=1-2, to=2-2]
	\arrow[from=1-1, to=1-2]
	\arrow[from=1-1, to=2-1]
	\arrow[dotted, from=2-1, to=1-2]
	\arrow[shift left=5, squiggly, tail reversed, from=2-3, to=2-4]
	\arrow[from=1-6, to=2-6]
	\arrow[from=1-5, to=1-6]
	\arrow[from=2-5, to=2-6]
	\arrow[from=1-5, to=2-5]
	\arrow[dotted, from=2-5, to=1-6]
\end{tikzcd}\]
and a correspondence of lifts given by the adjoint correspondence $\Hom(B\times D,X)\cong \Hom(D,\Map(B,X))$.

Given two simplicial sets $X$ and $Y$, we define their \emph{join} $X\star Y$ by $(X\star Y)_n=\coprod_{i+j=n-1} (X_i\times Y_j)$ (where $-1\leq i\leq n$ and we consider $X_{-1}$ and $Y_{-1}$ to be one point sets). The $d_{\ell}$ and $s_{\ell}$ maps out of $(X\star Y)_n$ correspond to $d_i$ and $s_i$ of $X_i$ if $\ell \leq i$ and correspond to $d_{\ell-i-1}$ and $s_{\ell-i-1}$ of $Y_j$ if $\ell\geq i+1$. For example, the join $\Delta[0]\star X$ contains a copy of $\Delta[0]$, a copy of $X$, and an $(n+1)$-simplex $\sigma'$ corresponding to each $n$-simplex $\sigma$ of $X$ where the 0th vertex of $\sigma'$ is $\Delta[0]$ and $d_0 \sigma'=\sigma$. Given $n,k\geq 0$, we have $\Delta[n]\star \Delta[k]\cong \Delta[n+k+1]$, where $\Delta[n]$ includes as the full subcomplex on $\{0,1,\ldots, n\}$ and $\Delta[k]$ includes as the full subcomplex on $\{n+1,n+2,\ldots,n+k+1\}$.

Given monomorphisms $A\hookrightarrow B$ and $C\hookrightarrow D$ of simplicial sets, the monomorphism
\[
(D\star A)\cup(C\star B)\hookrightarrow D\star B
\]
is called the \emph{pushout-join} of $A\hookrightarrow B$ and $C\hookrightarrow D$, which we denote by $(A\hookrightarrow B)\star (C\hookrightarrow D)$.

For a detailed overview of pushout-products and joins, see \cite{Rezk}.

\section{2-Segal horns}\label{sec:2segalhorns}

In this section, we define 2-Segal horns and show that they are necessarily weak equivalences in any model structure on $\sSet$ where the 2-Segal spine extensions are weak equivalences. We define quasi-2-Segal sets to be simplicial sets with fillers of 2-Segal horn inclusions, in analogy with quasi-categories.

\subsection{Fundamental definitions}

Recall that an ordinary horn $\Lambda^i[n]$ is the union of all but the $d_i$ face of a simplex $\Delta[n]$, and that a horn $\Lambda^i[n]$ is an inner horn if $0<i<n$. We define a 2-Segal horn to be the union of all but two faces of a simplex, where we need a slightly more complicated condition on which faces are missing, using the following definition.
\begin{defn}\label{def:broken}
Given $n\geq 3$, a subset $S\cin \{0,\ldots,n\}$ is \emph{broken} if there exist $0\leq i<j<k<\ell\leq n$ such that either $i$ and $k$ are in $S$ and $j$ and $\ell$ are not, or vice versa.
\end{defn}
The idea behind Definition \ref{def:broken} is that a subset $S\cin \{0,\ldots,n\}$ is broken if it has a gap modulo $n+1$, as illustrated in the following example. Note that a subset $S\cin \{0,1,\ldots n\}$ is broken if and only if its complement $\{0,1,\ldots n\}\smallsetminus S$ is broken in $\{0,1,\ldots n\}$.

\begin{ex}
Figure \ref{fig:broken} shows two examples of subsets $S\cin \{0,1,\ldots,7\}$, where in (a) $S$ is not broken but in (b) it is.
\begin{figure}[h]
\caption{\quad \quad \quad\quad\quad\quad}
\label{fig:broken}
\vspace{3mm}
\adjustbox{scale=1}{




\begin{tikzcd}[row sep=tiny, column sep=tiny]
	& 5 & 4 &&&&& {\color{red}{\mathbf{(5)}}} & {\color{red}{\mathbf{(4)}}} \\
	{\color{red}{\mathbf{(6)}}} &&& 3 &\hspace{3mm}&& 6 &&& {\color{red}{\mathbf{(3)}}} \\
	{\color{red}{\mathbf{(7)}}} &&& 2 &&& 7 &&& 2 \\
	& {\color{red}{\mathbf{(0)}}} & {\color{red}{\mathbf{(1)}}} &&&&& {\color{red}{\mathbf{(0)}}} & {\color{red}{\mathbf{(1)}}} \\
	& {} & {} &&&&& {} & {}
	\arrow["{\text{(a) Not Broken}}"{description}, draw=none, from=5-2, to=5-3]
	\arrow["{\text{(b) Broken}}"{description}, draw=none, from=5-8, to=5-9]
\end{tikzcd}
}
\end{figure}
\end{ex}

\begin{defn}
Given $n\geq 3$ and a broken subset $S\cin \{0,\ldots, n\}$, we say that the generalized horn $\Lambda^S[n]$ is a \emph{generalized 2-Segal horn}. If $\abs{S}=n-1$, then we say that $\Lambda^S[n]$ is a \emph{2-Segal horn}. An inclusion is \emph{2-Segal anodyne} if it is in the saturated class generated by the 2-Segal horn inclusions.
\end{defn}

Note that the condition $\abs{S}=n-1$ implies that a 2-Segal horn is the union of all but two of the faces of an $n$-simplex (which has a total of $n+1$ faces).

\begin{rmk}
Recall that we use two notations for generalized horns, depending on context. The notation $\Lambda^S[n]$, for $S\cin \{0,1,\ldots,n\}$, indicates that $\Lambda^S[n]$ is the union of the $d_{\ell}$ faces for $\ell\in S$, whereas the notation $\Lambda^{i_1,\ldots,i_k}[n]$ indicates that the $d_{i_1},\ldots,d_{i_k}$ faces are precisely the missing faces. It is therefore often more convenient to use the latter notation and denote 2-Segal horns by $\Lambda^{i,j}[n]$, for broken $\{i,j\}\cin \{0,1,\ldots n\}$.
\end{rmk}

Recall that the lowest dimensional inner horn $\Lambda^1[2]$ is precisely the 2-spine. We see that the analogous thing happens for 2-Segal horns in the following example.

\begin{ex}
The two broken subsets of $\{0,1,2,3\}$ are $\{0,2\}$ and $\{1,3\}$. The corresponding 2-Segal horns $\Lambda^{0,2}[3]$ and $\Lambda^{1,3}[3]$ are precisely the two triangulations of the square, our two 2-Segal spines from Subsection \ref{sub:2segspine}.
\end{ex}

Having defined 2-Segal horns as the 2-Segal analogue of inner horns, we define quasi-2-Segal sets to be simplicial sets with the corresponding filling condition, in analogy with quasi-categories which are defined in terms of having fillers of inner horns.

\begin{defn}\label{def:q2seg}
A simplicial set $X$ is a \emph{quasi-2-Segal set} if it has fillers of all 2-Segal horn inclusions.
\end{defn}

\begin{ex}
Every quasi-category is a quasi-2-Segal set, because every 2-Segal horn inclusion is a generalized inner horn inclusion. Thinking of quasi-categories as ``quasi-1-Segal sets,'' we recover the general fact that 1-Segal implies 2-Segal.
\end{ex}

\begin{ex}
In \cite{GKT:partI}, G\'alvez-Kock-Tonks define a 2-Segal groupoid $\mathbf{H}$ where $H_0$ is a point and $H_k$ for $k>1$ is the groupoid of forests of rooted trees with $k-1$ admissible cuts. Let $\mathbf{H}^{\operatorname{set}}$ be the simplicial set where $\mathbf{H}^{\operatorname{set}}_k$ is the set of objects of $H_k$ for each $k\geq 0$. Given a 2-Segal horn $\Lambda^{i,j}[k]\to \mathbf{H}^{\operatorname{set}}$, we can recover the data of a $k$-simplex as follows: each of the inner faces of the horn ($d_{\ell}$ for $0<\ell<k$ not equal to $i$ or $j$) specifies the underlying forest and $k-2$ of the cuts, while an outer face (either $d_0$ if $i>0$ or $d_n$ if $j<n$) specifies the remaining cut. Any remaining faces must agree with the cuts already specified in order to form the horn $\Lambda^{i,j}[k]\to \mathbf{H}^{\operatorname{set}}$ in the first place. We therefore see that $\mathbf{H}^{\operatorname{set}}$ is a quasi-2-Segal set.
\end{ex}

A primary goal of this paper is to justify these definitions by proving various fundamental results about 2-Segal horns and quasi-2-Segal sets. For example, we see in Section \ref{sec:pathspacecriterion} that another sensible definition in terms of path spaces turns out to be equivalent, and in Section \ref{sec:modelstructure} we see that there is an associated model structure on $\sSet$ in which quasi-2-Segal sets are precisely the fibrant objects. We devote the remainder of this section to showing that 2-Segal spine inclusions ``generate'' the 2-Segal horn inclusions, in the sense that assuming all 2-Segal spines in a Cisinski model structure on $\sSet$ are weak equivalences implies that all generalized 2-Segal horn inclusions must be weak equivalences as well. As a consequence, the model structure we construct in Section \ref{sec:modelstructure} is in fact the localization of the minimal model structure at the set of 2-Segal spine extensions.

\subsection{From 2-Segal spines to 2-Segal horns}

We begin by making a couple of observations about triangulations of the $(n+1)$-gon.

\begin{defn}
Given $n\geq 2$ and a triangulation $\mcT$ of the $(n+1)$-gon, we say that a vertex $0<i<n$ is \emph{extreme} if the triangle $(i-1)\to i\to (i+1)$ is in $\mcT$. We say that the vertex $0$ is extreme if the triangle $0\to 1\to n$ is in $\mcT$ and that the vertex $n$ is extreme if the triangle $0\to (n-1)\to n$ is in $\mcT$.
\end{defn}

\begin{lem}\label{lem:twoextreme}
Given $n\geq 2$, every triangulation $\mcT$ of the $(n+1)$-gon has at least two extreme vertices, at least one of which is not $0$ or $n$.
\end{lem}

\begin{proof}
Let us proceed by induction on $n$. If $n=2$, then every vertex of $\mcT=\Delta[2]$ is extreme. Now suppose that for some $n>2$, the hypothesis holds for every triangulation of the $(k+1)$-gon for every $2\leq k< n$. Given a triangulation $\mcT$ of the $(n+1)$-gon, let $i$ be the vertex such that $0\to i\to n$ is a triangle in $\mcT$. Then we can subdivide $\mcT$ into three pieces
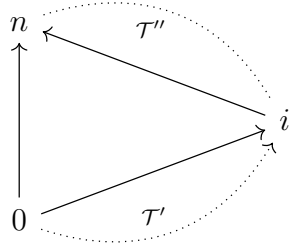
\[\begin{tikzcd}
	n \\
	&&& i \\
	0
	\arrow[from=3-1, to=2-4]
	\arrow[from=2-4, to=1-1]
	\arrow[from=3-1, to=1-1]
	\arrow["\mcT''", curve={height=30pt}, dotted, no head, from=2-4, to=1-1]
	\arrow["\mcT'", curve={height=30pt}, dotted, from=3-1, to=2-4]
\end{tikzcd}\]
where $\mcT'$ is trivial if $i=1$ and $\mcT''$ is trivial if $i=n-1$. If $\mcT'$ is nontrivial, then by induction it has at least one extreme vertex $0<j<i$, and hence that vertex is also extreme in $\mcT$. A similar argument shows that there is an extreme vertex $i<\ell<n$ if $\mcT''$ is nontrivial. In the case that $\mcT'$ is trivial because $i=1$, then $0$ itself is an extreme vertex of $\mcT$, and similarly $n$ is extreme in the case that $\mcT''$ is trivial because $i=n-1$. In any case, there are at least two extreme vertices, one of which is not $0$ or $n$, proving the claim.
\end{proof}

The above observation has the following converse of sorts, which says that any two non-adjacent vertices of the $(n+1)$-gon are extreme vertices of some triangulation.

\begin{lem}\label{lem:existstri}
Given $n\geq 3$ and $i<j$ such that $\{i,j\}\cin \{0,1,\ldots,n\}$ is broken, there exists some triangulation $\mcT$ of the $(n+1)$-gon in which the vertices $i$ and $j$ are extreme.
\end{lem}

\begin{proof}
For $n=3$, we note that the two broken subsets of $\{0,1,2,3\}$ are $\{0,2\}$ and $\{1,3\}$, and that these are the extreme vertices in the triangulations of the square $\Lambda^{1,3}[3]$ and $\Lambda^{0,2}[3]$, respectively. For $n>3$, we begin with the $(n+1)$-gon and add the edge $(i-1)\to (i+1)$ (or $1\to n$ if $i=0$) as well as the edge $(j-1)\to (j+1)$ (or $0\to (n-1)$ if $j=n$). We can then triangulate the remaining $(n-1)$-gon however we like, resulting in a triangulation in which the edges $i$ and $j$ are extreme.
\end{proof}

We are now ready to begin proving statements relating 2-Segal spine extensions to generalized 2-Segal horns. The following few results, culminating in Proposition \ref{prop:generalizedhornsareweakequivs}, are essentially a 2-Segal analogue of \cite[Lem.~3.5]{JT}.

\begin{lem}
Given a Cisinski model structure $\mcM$ on $\sSet$, suppose that every 2-Segal spine extension is a weak equivalence. Then, given $n\geq 3$ and a triangulation $\mcT$ of the $(n+1)$-gon with extreme vertex $0\leq i\leq n$, the inclusion
\[
\begin{tikzcd}
\mcT\cup\Delta[n-1] \ar[r,hookrightarrow,"\mcT\cup d^i"] & \Delta[n]
\end{tikzcd}
\]
is also a weak equivalence.
\end{lem}

\begin{proof}
Because the vertex $i$ is extreme in $\mcT$, the triangulation restricts along $d^i\colon \Delta[n-1]\hookrightarrow \Delta[n]$ to a triangulation $\mcT'\hookrightarrow \Delta[n-1]$. We can therefore factor the weak equivalence $\mcT\hookrightarrow \Delta[n]$ as
\[
\mcT\hookrightarrow \mcT\cup\Delta[n-1]\hookrightarrow \Delta[n]
\]
where the inclusion $\mcT\hookrightarrow \mcT\cup \Delta[n-1]$ is a pushout of the 2-Segal spine inclusion $\mcT'\hookrightarrow \Delta[n-1]$ and hence is a weak equivalence. By the 2-out-of-3 property, the remaining map is also a weak equivalence.
\end{proof}

\begin{lem}
Given a Cisinski model structure $\mcM$ on $\sSet$, suppose that every 2-Segal spine extension is a weak equivalence. Then, given $n\geq 3$ and $0\leq i<j\leq n$ such that $\{i,j\}\cin \{0,1,\ldots,n\}$ is broken, the generalized 2-Segal horn inclusion $\Lambda^{\{i,j\}}[n]\hookrightarrow \Delta[n]$ is also a weak equivalence.
\end{lem}

\begin{proof}
For $n=3$, the 2-Segal horn inclusions are automatically weak equivalences because they are also 2-Segal spine inclusions, so we may assume $n>3$. By Lemma \ref{lem:existstri}, there exists a triangulation $\mcT$ of the $(n+1)$-gon in which the vertices $i$ and $j$ are extreme. The inclusion
\[
\begin{tikzcd}
\mcT\cup\Delta[n-1] \ar[r,hookrightarrow,"\mcT\cup d^i"] & \Delta[n],
\end{tikzcd}
\]
which is a weak equivalence by the previous lemma, can be factored as
\[
\mcT\cup \Delta[n-1]\longhookrightarrow \Lambda^{\{i,j\}}[n]\hookrightarrow \Delta[n],
\]
where the first inclusion is a pushout of the 2-Segal spine inclusion
\[
\begin{tikzcd}
\mcT'\cup\Delta[n-2] \ar[r,hookrightarrow,"\mcT'\cup d^i"] & \Delta[n-1],
\end{tikzcd}
\]
where $\mcT'$ is the restriction of $\mcT$ along $d^j\colon \Delta[n-1]\hookrightarrow \Delta[n]$ (which is again a triangulation because $j$ is extreme). The remaining map $\Lambda^{\{i,j\}}[n]\hookrightarrow \Delta[n]$ is then a weak equivalence by the 2-out-of-3 property.
\end{proof}

Having shown that generalized 2-Segal horn inclusions of the form $\Lambda^{\{i,j\}}[n]\hookrightarrow \Delta[n]$ are necessarily weak equivalences in a Cisinski model structure where the 2-Segal spine extensions are weak equivalences, we may now inductively show that all generalized 2-Segal horn inclusions must be weak equivalences in the following lemma.

\begin{prop}\label{prop:generalizedhornsareweakequivs}
Given a Cisinski model structure $\mcM$ on $\sSet$, suppose that every 2-Segal spine extension is a weak equivalence. Then every generalized 2-Segal horn inclusion $\Lambda^{S}[n]\hookrightarrow \Delta[n]$ is a weak equivalence in $\mcM$.
\end{prop}

\begin{proof}
Let $N=\abs{S}$. We proceed by induction on $N$, with the base case $N=2$ being precisely the scenario of the previous lemma. Therefore, given $N\geq 2$, suppose that every generalized 2-Segal horn inclusion $\Lambda^{S}[n]\hookrightarrow \Delta[n]$ where $\abs{S}=N$ and $n\geq 3$ is arbitrary is a weak equivalence. Given a generalized 2-Segal horn inclusion  $\Lambda^{S}[n]\hookrightarrow \Delta[n]$ for some $n\geq 4$ with $\abs{S}=N+1$, we would like to show that it is also a weak equivalence. Because $\abs{S}>2$, there exists a subset $T\cin S$ with $\abs{T}=N$ where $T$ is broken in $\{0,1,\ldots,n\}$. Let $i$ be the element of $S$ such that $T\cup \{i\}=S$, and denote by $T'$ the preimage $(d^i)^{-1}(T)\cin \{0,1,\ldots,n-1\}$ which is also broken because $T$ is. We can therefore factor the inclusion
\[
\Lambda^{T}[n]\hookrightarrow \Delta[n],
\]
which is a weak equivalence by the inductive hypothesis, as
\[
\Lambda^{T}[n]\hookrightarrow \Lambda^{S}[n]\hookrightarrow \Delta[n],
\]
where the first inclusion is a pushout of the inclusion
\[
\Lambda^{T'}[n-1]\hookrightarrow \Delta[n-1]
\]
which is also a weak equivalence by the inductive hypothesis. By the 2-out-of-3 property the remaining inclusion $\Lambda^{S}[n]\hookrightarrow \Delta[n]$ must be a weak equivalence as well.
\end{proof}

\subsection{2-Segal horns as generators}
Now that we see that (generalized) 2-Segal horn inclusions are necessarily weak equivalences in a Cisinski model structure where the 2-Segal spine inclusions are weak equivalences, our remaining goal is to prove the converse by showing that the 2-Segal horn inclusions are sufficient to build all generalized 2-Segal horn inclusions and 2-Segal spine extensions. We begin with the generalized 2-Segal horn inclusions.

\begin{prop}\label{prop:generalizedhornsarepushouts}
Every generalized 2-Segal horn inclusion $\Lambda^S[n]\hookrightarrow \Delta[n]$ is a composite of pushouts of 2-Segal horn inclusions.
\end{prop}

\begin{proof}
We proceed by induction on $N=n+1-\abs{S}$, the number of missing faces of $\Lambda^S[n]$. The base case is $N=2$, which is when $\Lambda^S[n]\hookrightarrow \Delta[n]$ is already a 2-Segal horn inclusion. So, assuming the inductive hypothesis for $N\geq 2$, we would like to show that a given generalized 2-Segal horn $\Lambda^S[n]\hookrightarrow \Delta[n]$ for which $n+1-\abs{S}=N+1$ (i.e., the number of missing faces is $N+1$) is a composite of pushouts of 2-Segal horn inclusions. We do so by showing it is a composite of pushouts of generalized 2-Segal horn inclusions where the number of missing faces is equal to $N$.

Since $n+1-\abs{S}>2$, there exists $0\leq i\leq n$ not in $S$ such that $S\cup\{i\}$ is broken in $\{0,1,\ldots,n\}$. Let $T=S\cup \{i\}$, and denote by $S'$ the preimage $(d^i)^{-1}(S)\cin \{0,1,\ldots,n-1\}$ which is also broken because $S$ is. We can therefore factor the inclusion
\[
\Lambda^{S}[n]\hookrightarrow \Delta[n]
\]
as
\[
\Lambda^{S}[n]\hookrightarrow \Lambda^{T}[n]\hookrightarrow \Delta[n],
\]
where the first inclusion is a pushout of the inclusion
\[
\Lambda^{S'}[n-1]\hookrightarrow \Delta[n-1].
\]
The number of missing faces is $N$ for both $\Lambda^{S'}[n-1]$ and $\Lambda^{T}[n]$, as desired.
\end{proof}

We now check that the 2-Segal spine inclusions can be built out of 2-Segal horn inclusions.

\begin{lem}
Given $n\geq 3$ and a triangulation $\mcT$ of the $(n+1)$-gon with extreme vertex $i$, both the 2-Segal spine inclusion $\mcT\hookrightarrow \Delta[n]$ and the inclusion $\mcT\cup d^i\colon \mcT\cup \Delta[n-1]\hookrightarrow \Delta[n]$ are composites of pushouts of 2-Segal horn inclusions.
\end{lem}

\begin{proof}
We induct on $n$, with the base case $n=3$ following from the observation that the 2-Segal spine inclusions $\mcT\hookrightarrow \Delta[3]$ are already 2-Segal horn inclusions, and that if $i$ is an extreme vertex of $\mcT$ in this case then $d_i\Delta[n]$ is contained in $\mcT$, so that the inclusion $\mcT\cup d^i$ is the same as $\mcT\hookrightarrow \Delta[n]$.

Now assume that the hypothesis holds for some $n\geq 3$. Given a triangulation $\mcT$ of the $(n+2)$-gon with extreme vertex $i$, there must also be another extreme vertex $j$ by Lemma \ref{lem:twoextreme}. We factor $\mcT\hookrightarrow \Delta[n+1]$ as
\[
\mcT\longhookrightarrow \mcT \cup \Delta[n]\longhookrightarrow \Lambda^{\{i,j\}}[n+1]\longhookrightarrow \Delta[n+1],
\]
where the inclusion $\mcT\cup \Delta[n]\hookrightarrow \Delta[n+1]$ is $\mcT\cup d^i$. The last inclusion is a generalized 2-Segal horn inclusion (because $i$ and $j$ being extreme implies that they cannot be adjacent on the $(n+2)$-gon), and so is a composite of pushouts of 2-Segal horn inclusions by the previous lemma.

The inclusion $\mcT\cup \Delta[n]\hookrightarrow \Lambda^{\{i,j\}}[n+1]$ is a pushout of the inclusion
\[
\begin{tikzcd}
{\mcT'\cup \Delta[n-1]}\ar[r,"\mcT\cup d^{i'}"] & {\Delta[n],}
\end{tikzcd}
\]
where $\mcT'$ is the restriction of $\mcT$ to the $d_j$ face of $\Delta[n+1]$, and where $i'$ is the extreme vertex of $\mcT'$ corresponding to the vertex $i$ of $\mcT$. By the inductive hypothesis, this inclusion is a composite of pushouts of 2-Segal horn inclusions, and hence so is the composite $\mcT\cup \Delta[n]\hookrightarrow \Delta[n+1]$.

The inclusion $\mcT\hookrightarrow \mcT\cup \Delta[n]$ is a pushout of the inclusion $\mcT'\hookrightarrow \Delta[n]$ and so is also a composite of pushouts of 2-Segal horn inclusions by the inductive hypothesis. Therefore, the entire composite inclusion $\mcT\hookrightarrow \Delta[n+1]$ is such a composite of pushouts as well.
\end{proof}

For clarity, let us restate the main takeaway of the above lemma in the following corollary.

\begin{cor}\label{cor:spines}
Every 2-Segal spine inclusion is a composite of pushouts of 2-Segal horn inclusions.
\end{cor}

Because quasi-2-Segal sets are defined as simplicial sets with 2-Segal horn extensions, Proposition \ref{prop:generalizedhornsarepushouts} and Corollary \ref{cor:spines} have the following corollary.

\begin{cor}
Every quasi-2-Segal set has fillers of all 2-Segal spine inclusions and generalized 2-Segal horn inclusions.
\end{cor}

The above results are 2-Segal analogues of those about ordinary spines and (generalized) inner horn inclusions first proved by Joyal \cite[52.6]{Joyal:notes} \cite[Prop.~2.12 \& 2.13]{Joyal:theory}. Putting them together, we get the following proposition.

\begin{prop}\label{prop:spinesandhornsweakequivalences}
Given a Cisinski model structure $\mcM$ on $\sSet$, the following are equivalent:
\begin{enumerate}
    \item Every 2-Segal spine inclusion is a weak equivalence in $\mcM$.
    \item Every 2-Segal horn inclusion is a weak equivalence in $\mcM$.
    \item Every generalized 2-Segal horn inclusion is a weak equivalence in $\mcM$.
\end{enumerate}
\end{prop}

\begin{proof}
That (3) follows from (1) is the content of Proposition \ref{prop:generalizedhornsareweakequivs}. It is immediate that (2) follows from (3) because every 2-Segal horn is also a generalized 2-Segal horn. That (1) follows from (2) uses Corollary \ref{cor:spines} plus the observation that all of the maps in question are cofibrations (being monomorphisms) and hence are trivial cofibrations if they are weak equivalences. Any composite of pushouts of trivial cofibrations is again a trivial cofibration, hence a weak equivalence.
\end{proof}

\begin{rmk}\label{rmk:2segisquasi2seg}
One can show using arguments similar to this section's that a simplicial set $X$ has unique 2-Segal spine extensions if and only if it has unique (generalized) 2-Segal horn inclusions, and hence that every 2-Segal set is a quasi-2-Segal set.
\end{rmk}

\section{Path space criterion}\label{sec:pathspacecriterion}

One of the fundamental results of 2-Segal theory is that a simplicial space is 2-Segal if and only if its path spaces (also known as its upper and lower d\'ecalage) are Segal spaces, as first proved independently in \cite[Thm.~6.3.2]{DK} and \cite[Thm.~4.10]{GKT:partI}. The goal of this section is to show that the analogous fact is true of quasi-2-Segal sets: a simplicial set is quasi-2-Segal if and only if its path spaces are quasi-categories. We then see that this criterion gives us access to a wide class of examples of quasi-2-Segal sets.

\subsection{Path space criterion}
We begin by recalling the definition of the path spaces of a simplicial set.

\begin{defn}
Given a simplicial set $X$, the \emph{left path space} $P^{\triangleleft}(X)$ is the simplicial set whose $n$-simplices are maps $\Delta[0]\star\Delta[n]\to X$, and the \emph{right path space} $P^{\triangleright}(X)$ has $n$-simplices given by maps $\Delta[n]\star\Delta[0]\to X$.
\end{defn}

\begin{rmk}
In both cases, the $n$-simplices of the path space are $(n+1)$-simplices of $X$ because the join of $\Delta[0]$ and $\Delta[n]$ is $\Delta[n+1]$. However, the face and degeneracy maps are inherited from $X$ differently depending on which side we take the join.
\end{rmk}

Let us unpack what it means for the left path space $P^{\triangleleft}(X)$ to be a quasi-category. Given $n\geq 2$, an $n$-simplex in the path space $\Delta[n]\to P^{\triangleleft}(X)$ is an $(n+1)$-simplex $\Delta[n+1]\to X$, and the $j$th face of the $n$-simplex in the path space corresponds to the $(j+1)$ face of the $(n+1)$-simplex in $X$. Therefore, an inner horn in the path space $\Lambda^i[n]\to P^{\triangleleft}(X)$ corresponds to the 2-Segal horn $\Lambda^{0,i+1}[n+1]\to X$. Notice that $0<i<n$ is precisely what we needed for ${0,i+1}$ to be broken in $\{0,1,\ldots,n+1\}$. We record this observation and the corresponding one for right path spaces in the following lemma.

\begin{lem}
Given a simplicial set $X$, the left path space $P^{\triangleleft}(X)$ is a quasi-category if and only if $X$ has fillers of 2-Segal horns of the form $\Lambda^{0,j}[n]$, and the right path space $P^{\triangleright}(X)$ is a quasi-category if and only if $X$ has fillers of 2-Segal horns of the form $\Lambda^{i,n}[n]$.
\end{lem}

Another way to phrase this observation is that the set of 2-Segal horn inclusions missing the $d_0$ face
\[
\{\Lambda^{0,j}[n]\hookrightarrow \Delta[n]\}
\]
is isomorphic to $\Delta[0]\star \InnHorn$. Similarly, the set of generalized 2-Segal horn inclusions missing the $d_0$ face is isomorphic to $\Delta[0]\star \GenInnHorn$. Because the generalized inner horn inclusions are in the saturated class generated by the inner horn inclusions, we also have
\[
\Delta[0]\star \GenInnHorn \cin \ol{\Delta[0]\star \InnHorn},
\]
as well as the dual statement taking the join on the other side, so we can strengthen this statement as in the following corollary.

\begin{cor}
Given a simplicial set $X$, the left path space $P^{\triangleleft}(X)$ is a quasi-category if and only if $X$ has fillers of generalized 2-Segal horns of the form $\Lambda^{0,j_1,j_2,\ldots,j_r}[n]$, and the right path space $P^{\triangleright}(X)$ is a quasi-category if and only if $X$ has fillers of 2-Segal horns of the form $\Lambda^{i_1,i_2,\ldots,i_r,n}[n]$.
\end{cor}

The path spaces of a simplicial set being quasi-categories is therefore an ostensibly weaker condition than our definition of quasi-2-Segal set, because it does not directly say that there are fillers of 2-Segal horns such as $\Lambda^{1,3}[4]$. However, we see in the following proposition that the remaining 2-Segal horns can be built from generalized 2-Segal horns of the form $\Lambda^{0,j_1,j_2,\ldots,j_r}[n]$ and $\Lambda^{i_1,i_2,\ldots,i_r,n}[n]$, so that the desired path space criterion does in fact hold.

\begin{prop}
Given $0<i<j-1<n-1$, the 2-Segal horn inclusion $\Lambda^{i,j}[n]\hookrightarrow \Delta[n]$ is a retract of a composite of pushouts of generalized 2-Segal horn inclusions which are each missing an outer face.
\end{prop}

\begin{proof}
We form the following retract diagram




\[
\begin{tikzcd}[column sep=small]
	{\Lambda^{i,j}[n]} && {\Lambda^{0,i,j,n,n+1}[n+1]\cup (d_0d_{n+1}\Delta[n+1]) \cup (d_nd_{n+1}\Delta[n+1])} && {\Lambda^{i,j}[n]} \\
	{\Delta[n]} && {\Delta[n+1]} && {\Delta[n]\nospace{.}}
	\arrow["{d^{n+1}}"', from=2-1, to=2-3]
	\arrow["{s^{n}}"', from=2-3, to=2-5]
	\arrow[hook, from=1-5, to=2-5]
	\arrow[hook, from=1-3, to=2-3]
	\arrow[hook, from=1-1, to=2-1]
	\arrow[from=1-1, to=1-3]
	\arrow[from=1-3, to=1-5]
\end{tikzcd}
\]
To see that the left-hand square is valid, observe that the $d_0$ and $d_n$ faces of $\Delta[n]$ correspond respectively to the $d_0 d_{n+1}$ face and $d_n d_{n+1}$ face of $\Delta[n+1]$, and for $0<k<n$ with $k\neq i,j$, the $d_k$ face of $\Delta[n]$ is sent to the $d_{n}$ face of the $d_k$ face of $\Delta[n+1]$, and so is in $d_k\Delta[n+1]\cin\Lambda^{0,i,j,n,n+1}[n+1]$.

For the right-hand square, the $d_0 d_{n+1}$ face and $d_n d_{n+1}$ face are respectively mapped back to the $d_0$ and $d_n$ faces of $\Delta[n]$, and for $0<k<n$ with $k\neq i,j$, the $d_k$ face of $\Delta[n+1]$ is collapsed onto the $d_k$ face of $\Delta[n]$, and so is in $\Lambda^{i,j}[n]$.

This middle map can be built out of generalized 2-Segal horns which are each missing an outer face, as shown in the following diagram:




\[
\adjustbox{scale=0.75}{
\begin{tikzcd}[column sep=small]
	{\Lambda^{0,i,j}[n]} && {\Delta[n]} \\
	{\Lambda^{0,i,j,n,n+1}[n+1]\cup (d_0d_{n+1}) \cup (d_nd_{n+1})} && {\Lambda^{0,i,j,n+1}[n+1]\cup(d_0d_{n+1})} && {\Lambda^{i,j,n+1}[n+1]} && {\Delta[n+1]} \\
	&& {\Lambda^{i-1,j-1,n}[n]} && {\Delta[n]\nospace{.}}
	\arrow[hook, from=2-5, to=2-7]
	\arrow[hook, from=2-3, to=2-5]
	\arrow[hook, from=2-1, to=2-3]
	\arrow[from=1-1, to=2-1]
	\arrow[hook, from=1-1, to=1-3]
	\arrow["{d^n}", from=1-3, to=2-3]
	\arrow["{d^0}"', from=3-5, to=2-5]
	\arrow[from=3-3, to=2-3]
	\arrow[from=3-3, to=3-5]
\end{tikzcd}
}
\]

\end{proof}

In other words, we have shown that the sets of inclusions $\Delta[0]\star \InnHorn$ and $\InnHorn\star \Delta[0]$ generate the class of 2-Segal anodyne maps, yielding the following theorem.

\begin{thm}[Path Space Criterion]\label{thm:pathspacecriterion}
A simplicial set $X$ is a quasi-2-Segal set if and only if its path spaces are quasi-categories.
\end{thm}

\subsection{Examples}

As a corollary of Theorem \ref{thm:pathspacecriterion}, we get a wide class of examples of quasi-2-Segal sets coming from 2-Segal spaces. Recall that we view a simplicial space $W$ as a grid whose $n$th column is the simplicial set $W_n$, which is a Kan complex if $W$ is vertical Reedy fibrant.

\begin{cor}\label{cor:rowsof2segalspaces}
Given a vertical Reedy fibrant 2-Segal space $W$, each row of $W$ is a quasi-2-Segal set.
\end{cor}

\begin{proof}
The left (resp.~right) path space of a row of $W$ is precisely the corresponding row of the left (resp.~right) path space of $W$, which is a Segal space by the path space criterion for 2-Segal spaces \cite[Thm.~6.3.2]{DK}. By \cite[Cor.~3.6]{JT}, each row of a Segal space is a quasi-category. Since each of the path spaces of a given row of $W$ are quasi-categories, each row of $W$ is a quasi-2-Segal set by Theorem \ref{thm:pathspacecriterion}.
\end{proof}

This corollary implies that for many constructions which output 2-Segal spaces, the corresponding construction where one simply takes a set of $n$-simplices instead of a Kan complex outputs quasi-2-Segal sets. In particular, our motivating example, the discrete $S_{\bullet}$ construction, indeed outputs quasi-2-Segal sets.

\begin{ex}(Discrete $S_{\bullet}$)
Recall that every quasi-category $X$ contains a maximal Kan complex $\core X\cin X$. Dyckerhoff-Kapranov \cite[Def.~7.3.1]{DK} define a simplicial space $S_{\bullet}(\mcC)$ for $\mcC$ an exact quasi-category by taking a full subcomplex
\[
S_{n}(\mcC) \cin \core (\Map(N(T_n),\mcC))
\]
for each $n\geq 0$, where $T$ is a certain cosimplicial object in $\Cat$ they define in \cite[\S 2.4]{DK}. This full subcomplex is taken to span the set of diagrams $N(T_n)\to \mcC$ satisfying axioms they call (WS1), (WS2), and (WS3). To get a discrete version of this construction, we can define
\[
S^{\operatorname{set}}_{n}(\mcC)\cin \Hom(N(T_n),\mcC)
\]
as the subset of diagrams satisfying those three axioms. This discrete construction $S^{\operatorname{set}}_{\bullet}(\mcC)$ is then the $0$th row of the original construction $S_{\bullet}(\mcC)$, and so is a quasi-2-Segal set by Corollary \ref{cor:rowsof2segalspaces} because $S_{\bullet}(\mcC)$ is a 2-Segal space \cite[Thm.~7.3.3]{DK}.
\end{ex}

\section{Pushout-products and pushout-joins of 2-Segal anodyne maps}\label{sec:pushoutprod}

A compelling feature of quasi-categories is their characterization in terms of ``having a contractible space of composites for every composable pair of morphisms.'' More accurately, a simplicial set $X$ is a quasi-category if and only if the map $\Map(\Delta[2],X)\to \Map(\Lambda^1[2],X)$ is a trivial fibration. This fact follows from the fact that the set $(\Lambda^1[2]\hookrightarrow \Delta[2])\boxprod \Bdry$ generates the class of inner anodyne morphisms together with the adjoint correspondence




\[
\begin{tikzcd}
	{(A\times \Delta[n])\cup (B\times \partial\Delta[n])} & X &&& {\partial\Delta[n]} & {\Map(B,X)} \\
	{B\times \Delta[n]} & \ast & {} & {} & {\Delta[n]} & {\Map(A,X)\nospace{,}}
	\arrow[from=2-1, to=2-2]
	\arrow[from=1-2, to=2-2]
	\arrow[from=1-1, to=1-2]
	\arrow[from=1-1, to=2-1]
	\arrow[dotted, from=2-1, to=1-2]
	\arrow[shift left=5, squiggly, tail reversed, from=2-3, to=2-4]
	\arrow[from=1-6, to=2-6]
	\arrow[from=1-5, to=1-6]
	\arrow[from=2-5, to=2-6]
	\arrow[from=1-5, to=2-5]
	\arrow[dotted, from=2-5, to=1-6]
\end{tikzcd}
\]
in the case $A\hookrightarrow B$ is the inner horn inclusion $\Lambda^1[2]\hookrightarrow \Delta[2]$. The goal of this section is to show that the sets $(\Lambda^{0,2}[3]\hookrightarrow \Delta[3])\boxprod \Bdry$ and $(\Lambda^{1,3}[3]\hookrightarrow \Delta[3])\boxprod \Bdry$ together generate the same saturated class as the 2-Segal horn inclusions. With that fact, we can apply the correspondence above to conclude that a simplicial set $X$ is quasi-2-Segal if and only if the maps $\Map(\Delta[3],X)\to \Map(\Lambda^{0,2}[3],X)$ and $\Map(\Delta[3],X)\to \Map(\Lambda^{1,3}[3],X)$ are trivial fibrations. 

To show that $(\Lambda^{0,2}[3]\hookrightarrow \Delta[3])\boxprod \Bdry\cup(\Lambda^{1,3}[3]\hookrightarrow \Delta[3])\boxprod \Bdry$ generates the class of 2-Segal anodyne maps, we use the following diagram




\[\begin{tikzcd}[column sep=small]
	\TwoSegHorn && {\left((\Lambda^{0,2}[3]\hookrightarrow \Delta[3])\boxprod \Bdry\right)\cup \left((\Lambda^{1,3}[3]\hookrightarrow \Delta[3])\boxprod \Bdry\right)} \\
	{\Delta[0]\star\InnHorn\cup\InnHorn\star\Delta[0]} && {\left((\Lambda^{0,2}[3]\hookrightarrow \Delta[3])\boxprod \Mono\right)\cup \left((\Lambda^{1,3}[3]\hookrightarrow \Delta[3])\boxprod \Mono\right)}
	\arrow["{(a)}"', from=1-1, to=2-1]
	\arrow["{(b)}"', dotted, from=2-1, to=2-3]
	\arrow["{(c)}"', from=2-3, to=1-3]
	\arrow["{(d)}"', dotted, from=1-3, to=1-1]
\end{tikzcd}\]
where an arrow $S\to T$ indicates that $S\cin \ol{T}$. We may conclude that all four sets generate the same saturated class once we prove the existence of each of these arrows. We note that (a) is the observation from Section \ref{sec:pathspacecriterion} that the set $\Delta[0]\star \InnHorn$ is the set of 2-Segal horn inclusions of the form $\Lambda^{0,j}[n]\hookrightarrow \Delta[n]$ and the corresponding statement for $\InnHorn\star\Delta[0]$, while (c) follows from the fact that $\ol{S}\boxprod \ol{T}\cin \ol{S\boxprod T}$ for any sets of maps $S$ and $T$ together with the fact that $\ol{\Bdry}=\Mono$. Our task for the remainder of this section is therefore to prove (b) and (d). We begin by proving (b).

\begin{prop}
Every 2-Segal horn inclusion of the form $\Lambda^{0,j}[n]\hookrightarrow \Delta[n]$ is a retract of the inclusion $(\Lambda^{0,2}[3]\hookrightarrow \Delta[3])\boxprod (\Lambda^{0,j}[n]\hookrightarrow \Delta[n])$, and every 2-Segal horn inclusion of the form $\Lambda^{i,n}[n]\hookrightarrow \Delta[n]$ is a retract of the inclusion $(\Lambda^{1,3}[3]\hookrightarrow \Delta[3])\boxprod (\Lambda^{i,n}[n]\hookrightarrow \Delta[n])$.
\end{prop}

\begin{proof}
By symmetry, it suffices to prove the first statement. Our retract diagram is




\[\begin{tikzcd}
	{\Lambda^{0,j}[n]} && {(\Lambda^{0,2}[3]\times\Delta[n])\cup(\Delta[3]\times\Lambda^{0,j}[n])} && {\Lambda^{0,j}[n]} \\
	{\Delta[n]} && {\Delta[3]\times\Delta[n]} && {\Delta[n]}
	\arrow[hook, from=1-1, to=2-1]
	\arrow[hook, from=1-3, to=2-3]
	\arrow[hook, from=1-5, to=2-5]
	\arrow[hook, from=1-1, to=1-3]
	\arrow["f"', hook, from=2-1, to=2-3]
	\arrow["p"', from=2-3, to=2-5]
	\arrow[from=1-3, to=1-5]
\end{tikzcd}\]
where we define $f$ and $p$ as the following poset maps:
\begin{equation*}
f(\ell)=\begin{cases}
          (0,\ell) \quad &\text{if } \, \ell=0\\
          (1,\ell) &\text{if } \, 0<\ell<j\\
          (2,\ell) \quad &\text{if } \, \ell=j\\
          (3,\ell) \quad &\text{if } \, \ell>j
     \end{cases}
\end{equation*}
and
\begin{equation*}
p(a,b)=\begin{cases}
          0 \quad &\text{if } \, a=0\\
          b \quad &\text{if } \, a=1 \text{ and }0\leq b\leq j\\
          j \quad &\text{if } \, a=1 \text{ and }b> j\\
          j \quad &\text{if } \, a=2\\
          j \quad &\text{if } \, a=3 \text{ and }b\leq j\\
          b \quad &\text{if } \, a=3 \text{ and }b> j\nospace{.}
     \end{cases}
\end{equation*}
A case-by-case check shows that $pf=\id_{\Delta[n]}$. We also have $f(\Lambda^{0,j}[n])\cin \Delta[3]\times\Lambda^{0,j}[n]$ since $f$ is the identity on the second component, justifying the lefthand square. It now remains to verify the righthand square by checking that the restriction of $p$ to $\Lambda^{0,2}[3]\times \Delta[n]$ and to $\Delta[3]\times\Lambda^{0,j}[n]$ lands in $\Lambda^{0,j}[n]\cin \Delta[n]$.

To see that $\Lambda^{0,2}[3]\times\Delta[n]$ lands in $\Lambda^{0,j}[n]$, we observe that
\[
p(d_3 \Delta[3]\times\Delta[n])\cin d_{\{0,1,\ldots,j\}} \Delta[n] \cin \Lambda^{0,j}[n]
\]
because $j<n$ (since $\{0,j\}\cin \{0,1,\ldots,n\}$ is broken). We also have
\[
p(d_1 \Delta[3]\times \Delta[n])\cin d_{\{0,j,j+1,\ldots,n\}} \Delta[n] \cin \Lambda^{0,j}[n]
\]
because $j>1$ (so $d_{\{0,j,j+1,\ldots,n\}} \Delta[n]\cin d_1 \Delta[n]$).

To see that $\Delta[3]\times \Lambda^{0,j}[n]$ lands in $\Lambda^{0,j}[n]$, first take $0<k<j$. We then have
\[
p(\Delta[3]\times d_k \Delta[n])\cin d_k \Delta[n] \cin \Lambda^{0,j}[n]
\]
because the only vertex of $\Delta[3]\times \Delta[n]$ which maps to $k$ is $(1,k)$ which is not in $\Delta[3]\times d_k \Delta[n]$. Now if we take $j<k\leq n$, we similarly have
\[
p(\Delta[3]\times d_k \Delta[n])\cin d_k \Delta[n] \cin \Lambda^{0,j}[n]
\]
because the only vertex of $\Delta[3]\times \Delta[n]$ which maps to $k$ is $(3,k)$ which is not in $\Delta[3]\times d_k \Delta[n]$.
\end{proof}

We now turn to proving (d), for which we first need a lemma about pushout-joins.

\begin{lem}\label{lem:pushoutjoin}
Given an anodyne inclusion $f$ and an inner anodyne inclusion $g$, their pushout-joins $f\star g$ and $g\star f$ are 2-Segal anodyne.
\end{lem}

\begin{proof}
By symmetry, it suffices to show that $f\star g$ is 2-Segal anodyne. Furthermore, it suffices to consider the special case where $f$ is a horn inclusion $\Lambda^i[n]\to \Delta[n]$ and $g$ is an inner horn inclusion $\Lambda^j[k]\hookrightarrow \Delta[k]$. In this case, the pushout-join
\[
(\Lambda^i[n]\star \Delta[k])\cup(\Delta[n]\star\Lambda^j[k])\longhookrightarrow \Delta[n]\star \Delta[k]
\]
turns out to be isomorphic to the 2-Segal horn inclusion
\[
\Lambda^{i,n+j+1}[n+k+1]\hookrightarrow \Delta[n+k+1].
\]
To see why, we start with the fact that $\Delta[n]\star\Delta[k]\cong \Delta[n+k+1]$. We then observe for $0\leq \ell\leq n$ that $d_{\ell}\Delta[n]\star \Delta[k]$ corresponds to the $d_{\ell}$ face of $\Delta[n+k+1]$, and similarly for $0\leq \ell\leq k$ that $\Delta[n]\star d_{\ell}\Delta[k]$ corresponds to the $d_{n+\ell+1}$ face of $\Delta[n+k+1]$. We therefore have that $\Lambda^i[n]\star \Delta[k]$ is the union of the $d_0$ to $d_n$ faces of $\Delta[n+k+1]$ except for $d_i$ and that $\Delta[n]\star\Lambda^j[k]$ is the union of the $d_{n+1}$ to $d_{n+k+1}$ faces of $\Delta[n+k+1]$ except for $d_{n+j+1}$, hence they together form $\Lambda^{i,n+j+1}[n+k+1]$, which is 2-Segal because $0<j<k$.
\end{proof}

\begin{rmk}
The argument in the previous lemma also works to show that if $f$ and $g$ are both left anodyne or both right anodyne then their pushout-join is 2-Segal anodyne, since $i$ and $n+j+1$ are not adjacent modulo $n+k+1$ in those cases as well.
\end{rmk}

We are now ready to prove (d). In fact, we prove a stronger statement, and (d) is the special case when $A\hookrightarrow B$ is $\Lambda^1[2]\hookrightarrow \Delta[2]$.

\begin{prop}
Given an inner anodyne map $A\hookrightarrow B$, the pushout-products
\[
(\partial\Delta[n]\hookrightarrow \Delta[n])\boxprod (\Delta[0]\star A \hookrightarrow \Delta[0]\star B)
\]
and
\[
(\partial\Delta[n]\hookrightarrow \Delta[n])\boxprod (A\star \Delta[0] \hookrightarrow B\star \Delta[0])
\]
are 2-Segal anodyne.
\end{prop}

\begin{proof}
By symmetry, it suffices to prove the first claim, so we would like to show that the inclusion
\[
(\partial \Delta[n]\times (\Delta[0]\star B))\cup (\Delta[n]\times (\Delta[0]\star A))\longhookrightarrow (\Delta[n]\times (\Delta[0]\star B))
\]
is 2-Segal anodyne. Let us denote these simplicial sets by $Y=(\partial \Delta[n]\times (\Delta[0]\star B))\cup (\Delta[n]\times (\Delta[0]\star A))$ and $Z=\Delta[n]\times (\Delta[0]\star B)$.

We begin by observing that $Z$ decomposes into a set of pieces $Z_0,Z_1,\ldots,Z_n$, where $Z_i$ is the full subcomplex spanning (the set of vertices of)
\[
(d_{\{0,1,\ldots,i\}}\Delta[n]\times \ast)\cup (d_{\{i,i+1,\ldots,n\}}\Delta[n]\times B).
\]

Here is a schematic of what these pieces are for $n=2$:




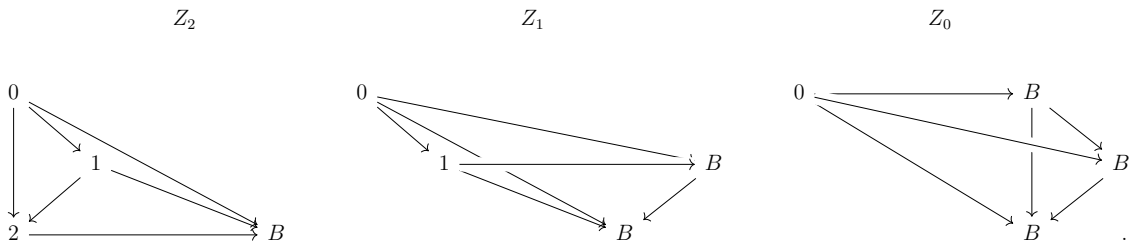
\[
\adjustbox{scale=0.7}{
\begin{tikzcd}
	&& {Z_2} &&&& {Z_1} &&&&& {Z_0} \\
	0 &&&& 0 &&&&& 0 &&& B \\
	& 1 & {} &&& 1 &&& B &&&&& B \\
	2 &&& B &&&& B &&&&& B & \ .
	\arrow[from=2-1, to=3-2]
	\arrow[from=3-2, to=4-1]
	\arrow[from=2-1, to=4-1]
	\arrow[from=3-2, to=4-4]
	\arrow[from=4-1, to=4-4]
	\arrow[from=2-1, to=4-4]
	\arrow[from=2-5, to=3-6]
	\arrow[from=3-9, to=4-8]
	\arrow[from=2-5, to=4-8]
	\arrow[from=2-5, to=3-9]
	\arrow[from=3-6, to=4-8]
	\arrow[from=2-10, to=2-13]
	\arrow[from=2-13, to=3-14]
	\arrow[from=3-14, to=4-13]
	\arrow[from=2-13, to=4-13]
	\arrow[from=2-10, to=4-13]
	\arrow[crossing over, from=3-6, to=3-9]
	\arrow[crossing over, from=2-10, to=3-14]
\end{tikzcd}
}
\]

For each $0\leq i<n$, denote by $W_{i+1,i}$ the intersection of $Z_i$ and $Z_{i+1}$, which is the full subcomplex spanning (the set of vertices of)
\[
(d_{\{0,1,\ldots,i\}}\Delta[n]\times \ast)\cup (d_{\{i+1,i+2,\ldots,n\}}\Delta[n]\times B).
\]
Note that for any other $0\leq i <j\leq n$, the intersection $Z_j\cap Z_i$ is contained in $W_{\ell+1,\ell}$ for each $i\leq \ell<j$, so in the following poset of subcomplexes of $\Delta[n]\times(\Delta[0]\star B)$



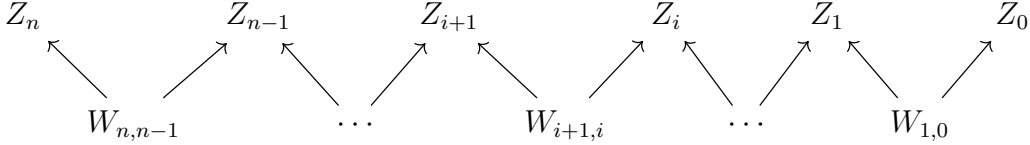
\[\begin{tikzcd}[column sep=tiny]
	{Z_n} && {Z_{n-1}} && {Z_{i+1}} && {Z_i} && {Z_1} && {Z_0} \\
	& {W_{n,n-1}} && \cdots && {W_{i+1,i}} && \cdots && {W_{1,0}}
	\arrow[from=2-2, to=1-1]
	\arrow[from=2-2, to=1-3]
	\arrow[from=2-4, to=1-3]
	\arrow[from=2-4, to=1-5]
	\arrow[from=2-6, to=1-5]
	\arrow[from=2-6, to=1-7]
	\arrow[from=2-8, to=1-7]
	\arrow[from=2-8, to=1-9]
	\arrow[from=2-10, to=1-9]
	\arrow[from=2-10, to=1-11]
\end{tikzcd}\]
each intersection $(Z_n\cup Z_{n-1}\cup \ldots Z_{i+1})\cap Z_i$ is still $W_{i+1,i}$.

We now observe that each $Z_i$ is isomorphic to $\Delta[i]\star (\Delta[n-i]\times B)$ while $W_{i+1,i}$ is isomorphic to $\Delta[i]\star (\Delta[n-i-1]\times B)$. We can more suggestively write $W_{i+1,i} \rightarrow Z_i\leftarrow W_{i,i-1}$ as



\[
\adjustbox{scale=0.73}{\begin{tikzcd}[column sep=tiny]
	&&& {(d_{\{0,\ldots,i\}}\Delta[n])\star(d_{\{i,\ldots,n\}}\Delta[n]\times B)} \\
	{} && {(d_{\{0,\ldots,i\}}\Delta[n])\star(d_{\{i+1,\ldots,n\}}\Delta[n]\times B)} && {(d_{\{0,\ldots,i-1\}}\Delta[n])\star(d_{\{i,\ldots,n\}}\Delta[n]\times B)\nospace{.}}
	\arrow[from=2-3, to=1-4]
	\arrow[from=2-5, to=1-4]
\end{tikzcd}}\]

We now use the breakdown above to describe how to build $Z$ from $Y=(\partial \Delta[n]\times (\Delta[0]\star B))\cup (\Delta[n]\times (\Delta[0]\star A))$. We do so by gluing in the missing part of each $Z_i$ one at a time, starting with $Z_n$ and working down to $Z_0$.

Let $Y_n$ denote the intersection of $Y$ with $Z_n$. We would like to see that the inclusion $Y_n\hookrightarrow Z_n$ is 2-Segal anodyne. In the case $i=n$, the isomorphism noted above reduces to $Z_n\cong \Delta[n]\star B$. The pieces of $Y_i$ from $\Delta[n]\times (d_{\{n\}}\Delta[n]\star A)$ correspond to $\Delta[n]\star A\cin \Delta[n]\star B$ and the pieces of $Y_i$ from $\partial \Delta[n]\times (d_{\{n\}}\Delta[n]\star B)$ correspond to $(d_n\Delta[n]\star B)\cup(\Lambda^n[n]\star B)$. To see the latter claim, take $0\leq \ell<n$, then $(d_{\ell}\Delta[n]\times B)\cap Z_n$ corresponds to $d_{\ell} \Delta[n]\star B$ in $\Delta[n]\star B$. Meanwhile, the intersection of $d_n\Delta[n]\times B$ with $Z_n$ is only $d_n\Delta[n]\cin \Delta[n]\star B$. Altogether, we see that $Y_n\hookrightarrow Z_n$ is isomorphic to the pushout-join $(\Lambda^n[n]\hookrightarrow \Delta[n])\star(A\hookrightarrow B)$, which we know is 2-Segal anodyne by Lemma \ref{lem:pushoutjoin} because $A\hookrightarrow B$ is inner anodyne.

Now, assume we have glued in $Z_n,Z_{n-1},\ldots, Z_{i+1}$ for some $0<i<n$. Let $Y_i$ denote the intersection of $Z_i$ with $Y\cup W_{i,i+1}$, which is precisely the part of $Z_i$ which has not been glued in yet. We claim that $Y_i\hookrightarrow Z_i$ is isomorphic to the inclusion of
\[
\left(\Lambda^i[i]\star (\Delta[n-i]\times B)\right)\cup\left(\Delta[i]\star (\partial\Delta[n-i]\times B)\right)\cup\left(\Delta[i]\star (\Delta[n-i]\times A)\right)
\]
into $\Delta[i]\star (\Delta[n-i]\times B)$, which is precisely the inclusion
\[
(\Lambda^i[i]\hookrightarrow \Delta[i])\star\left((\partial\Delta[n-i]\hookrightarrow\Delta[n-i])\boxprod (A\hookrightarrow B)\right),
\]
which is 2-Segal anodyne by Lemma \ref{lem:pushoutjoin} because $(\partial\Delta[n-i]\hookrightarrow\Delta[n-i])\boxprod (A\hookrightarrow B)$ is inner anodyne. Let us justify our claim. First, the part of $Z_i$ in $\Delta[n]\times(\Delta[0]\star A)$ corresponds to $\Delta[i]\star (\Delta[n-i]\times A)$. For $0\leq \ell <i$, the intersection of $Z_i$ with $d_{\ell} \Delta[n]\times (\Delta[0]\star B)$ corresponds to $d_{\ell}\Delta[i]\star (\Delta[n-i]\times B)$, which together form $\Lambda^i[i]\star (\Delta[n-i]\times B)$. For $i<\ell\leq n$, the intersection of $Z_i$ with $d_{\ell} \Delta[n]\times (\Delta[0]\star B)$ corresponds to $\Delta[i]\star (d_{\ell-i}\Delta[n-i]\times B)$, which together form $\Delta[i]\star (\Lambda^0 [n-i]\times B)$. The remaining piece is $\Delta[i]\star (d_{0}\Delta[n-i]\times B)$, which corresponds to $W_{i+1,i}$.

Suppose now that we have glued in each $Z_i$ except for $Z_0$, and let $Y_0$ denote the intersection of $Z_0$ with $Y\cup W_{0,1}$, which is precisely the part of $Z_0$ which has not been glued in yet. A similar check as for $i>0$ shows that $Y_0\hookrightarrow Z_0$ is isomorphic to
\[
\Delta[0]\star\left((\partial\Delta[n]\hookrightarrow \Delta[n])\boxprod (A\hookrightarrow B)\right).
\]
\end{proof}

We have now proved our desired proposition.

\begin{prop}\label{prop:genbypushoutprodoftriangs}
The class of 2-Segal anodyne maps is generated by the set
\[
\left((\Lambda^{0,2}[3]\hookrightarrow \Delta[3])\boxprod \Bdry\right)\cup \left((\Lambda^{1,3}[3]\hookrightarrow \Delta[3])\boxprod \Bdry\right),
\]
as well as by the class
\[
\left((\Lambda^{0,2}[3]\hookrightarrow \Delta[3])\boxprod \Mono\right)\cup \left((\Lambda^{1,3}[3]\hookrightarrow \Delta[3])\boxprod \Mono\right).
\]
\end{prop}

Using the adjoint correspondence discussed in the beginning of the section, we have the following corollary.

\begin{cor}\label{cor:contractiblespace}
A simplicial set $X$ is quasi-2-Segal if and only if the maps
\[
\Map(\Delta[3],X)\to \Map(\Lambda^{0,2}[3],X)\quad\text{ and }\quad \Map(\Delta[3],X)\to \Map(\Lambda^{1,3}[3],X)
\]
are trivial fibrations.
\end{cor}

The following corollary is one of our ingredients for producing a model structure in the next section.

\begin{cor}\label{cor:2seganodyneclosedunderpushoutprod}
The pushout-product of a 2-Segal anodyne map with a monomorphism is again 2-Segal anodyne, i.e.,
\[
\ol{\TwoSegHorn}\boxprod \Mono \cin \ol{\TwoSegHorn}.
\]
\end{cor}

\begin{proof}
Using Proposition \ref{prop:genbypushoutprodoftriangs}, we have
\[
\ol{\TwoSegHorn}\boxprod \Mono = \ol{\left((\Lambda^{0,2}[3]\hookrightarrow \Delta[3])\cup (\Lambda^{0,2}[3]\hookrightarrow \Delta[3])\right)\boxprod \Bdry}\boxprod \Mono,
\]
which is contained in
\[
\ol{\left((\Lambda^{0,2}[3]\hookrightarrow \Delta[3])\cup (\Lambda^{0,2}[3]\hookrightarrow \Delta[3])\right)\boxprod \Bdry\boxprod \Mono}.
\]
Since the pushout-product preserves monomorphisms, this class is contained in
\[
\ol{\left((\Lambda^{0,2}[3]\hookrightarrow \Delta[3])\cup (\Lambda^{0,2}[3]\hookrightarrow \Delta[3])\right)\boxprod \Mono},
\]
which is precisely $\ol{\TwoSegHorn}$ by Proposition \ref{prop:genbypushoutprodoftriangs}.
\end{proof}

As a further corollary, we see that mapping spaces of quasi-2-Segal sets are themselves quasi-2-Segal sets, generalizing another fundamental property of quasi-categories.

\begin{cor}
Given a simplicial set $X$ and a quasi-2-Segal set $Y$, the simplicial set $\Map(X,Y)$ is quasi-2-Segal.
\end{cor}

\begin{proof}
Given a 2-Segal anodyne map $A\hookrightarrow B$, the inclusion $A\times X\hookrightarrow B\times X$ is also 2-Segal anodyne by applying Corollary \ref{cor:2seganodyneclosedunderpushoutprod} with the inclusion $\varnothing \hookrightarrow X$. Therefore, given a lifting problem as on the right below,



\[
\begin{tikzcd}
	{A\times X} & Y &&& A & {\Map(X,Y)} \\
	{B\times X} & \ast & {} & {} & B & \ast\nospace{,}
	\arrow[from=2-1, to=2-2]
	\arrow[from=1-2, to=2-2]
	\arrow[from=1-1, to=1-2]
	\arrow[from=1-1, to=2-1]
	\arrow[dotted, from=2-1, to=1-2]
	\arrow[shift left=5, squiggly, tail reversed, from=2-3, to=2-4]
	\arrow[from=1-6, to=2-6]
	\arrow[from=1-5, to=1-6]
	\arrow[from=2-5, to=2-6]
	\arrow[from=1-5, to=2-5]
	\arrow[dotted, from=2-5, to=1-6]
\end{tikzcd}
\]
there is a lift because we can solve the adjoint lifting problem on the left.
\end{proof}

\section{The model structure for quasi-2-Segal sets}\label{sec:modelstructure}

In this section, we show that quasi-2-Segal sets have fillers of $J$-augmented horn inclusions, which is analogous to the special outer horn lifting property of quasi-categories. By combining this result with our previous results, we get a model structure for quasi-2-Segal sets. We then use our model structure to prove a quasi-2-Segal version of the edgewise subdivision criterion from \cite{BOORS:edgewise}.

\subsection{Augmented horn lifting}

Recall that $J$ is the nerve of the free-living isomorphism (which has two objects and exactly one morphism in every hom set), and that a $J$-augmented horn inclusion $\Lambda^i[n]^J_{j\to j+1}\hookrightarrow \Delta[n]^J_{j\to j+1}$ is an ordinary horn inclusion $\Lambda^i[n]\hookrightarrow \Delta[n]$ for $n\geq 2$ and $0\leq i\leq n$ with a copy of $J$ glued in along the $j\to j+1$ edge, where $j=i-1$ or $j=i$. We think of $\Delta[n]^J_{j\to j+1}$ as a homotopy from the $d_{j+1}$ face to the $d_j$ face, and that the $J$-augmented horn is a homotopy of the boundaries of those faces.

We begin by showing that in a quasi-2-Segal set we can ``invert'' a given homotopy $\Delta[n]^J_{j\to j+1}$, and in fact that we may (coherently) specify the inverse of any subset of its faces $d_k \Delta[n]^J_{j\to j+1}$ for $k\neq j,j+1$ (which are each isomorphic to $\Delta[n-1]^J_{j'\to j'+1}$).

\begin{lem}\label{lem:invertinghomotopy}
Let $X$ be a quasi-2-Segal set. Given $n\geq 2$ and an $n$-simplex $h$ of $X$ whose $j\to j+1$ edge is the $0\to 1$ edge of some $a\colon J\to X$, there exists an $(n+2)$-simplex $H$ where $d_{\{j,j+1,j+2,j+3\}}H$ is the 3-simplex $0\to 1\to 0\to 1$ of $a\colon J\to X$, and where $d_{j+1}H$ is $s_j h$ and $d_{j+2}H=s_{j+1}h$. Furthermore, we are free to specify any number of the faces $d_k H$ for $0\leq k<j$ and $j+3<k\leq n$ as long as they agree with the given data and with each other.
\end{lem}

\begin{proof}
We proceed by induction on $n$. For the base case $n=2$, we form a 2-Segal horn $\Lambda^{j, j+3}[4]\to X$ as follows: for $k\not\in \{j,j+1,j+2,j+3\}$ (so $k=0$ if $j=1$ and $k=4$ if $j=0$) we let $d_k\Lambda^{j, j+3}[4]$ be the $0\to 1\to 0\to 1$ face of $a\colon J\to X$, while we let $d_{j+1}\Lambda^{j, j+3}[4]$ be $s_j h$ and let $d_{j+2}\Lambda^{j, j+3}[4]=s_{j+1}h$. Because $X$ is quasi-2-Segal, this 2-Segal horn has a filler $H\colon \Delta[4]\to X$ which satisfies the desired conditions by construction because $d_k H$ is precisely $d_{\{j,j+1,j+2,j+3\}}H$ in this case, which also shows that the additional claim does not add anything for this case.

Now assume the hypothesis holds for some $n\geq 2$. Given an $(n+1)$-simplex $h$ of $X$ whose $j\to j+1$ edge is the $0\to 1$ edge of some $a\colon J\to X$, take any $0\leq k<j$ or $j+3<k\leq n$ and apply the inductive hypothesis to $d_k h$, yielding an $(n+2)$-simplex $H'$. We then form the generalized 2-Segal horn $\Lambda^{\{k,j+1,j+2\}}[n+3]\to X$ where the $d_k$ face is $H'$ and we let the $d_{j+1}$ face be $s_j h$ and let the $d_{j+2}$ face be $s_{j+1}h$. The filler $H\colon \Delta[n+3]\to X$ then satisfies the desired conditions by construction. To specify a certain subset of the faces of $H$ corresponding to a set of indices $S\cin \{0,1,\ldots,j-1,j+4,j+5,\ldots,n\}$, rather than apply the inductive hypothesis we instead form the generalized 2-Segal horn $\Lambda^{S\cup\{j+1,j+2\}}[n+3]\to X$ where the $d_{j+1}$ and $d_{j+2}$ are as before, but the $d_k$ faces for $k\in S$ are those we are specifying.
\end{proof}

\begin{prop}[$J$-augmented horn lifting]\label{prop:augmentedhornlifting}
Every quasi-2-Segal set has fillers of $J$-augmented horn inclusions.
\end{prop}

\begin{proof}
Let $X$ be a quasi-2-Segal set, and take a horn $\lambda\colon\Lambda^i[n]\to X$ for some $n\geq 2$, $0\leq i\leq n$, and whose $j\to j+1$ edge is the $0\to 1$ edge of some $a\colon J\to X$ for either $j=i-1$ or $j=i$. Without loss of generality assume $j=i$.

If $n=2$, again without loss of generality assume $j=0$. Then we can form the 2-Segal horn $\Lambda^{0,2}[3]\to X$ where the $d_3$ face is the 2-simplex $0\to 1\to 0$ of $J$ and the $d_1$ face is $s_1 d_1 \lambda$. This 2-Segal horn has a filler $\tau$ where $d_2 \tau$ is a filler for the original horn $\lambda$.

Now assume $n\geq 3$. Our goal is to construct a 2-Segal horn $\lambda'\colon\Lambda^{i,i+2}[n+1]\to X$ which restricts to $\Lambda^i[n]$ along $d^{i+2}$. First, we let the $d_{i+1}$ face be $s_i d_{i+1}\lambda$. Then, for each $k<i$, we apply the previous lemma to $d_k \lambda$ to get an $(n+1)$-simplex $H_k$ and let $d_k\lambda'=H_k$, and similarly for $k>i+1$ we apply the lemma to $d_k\lambda$ to get $H_k$ and let $d_{k+1}\lambda'=H_k$. We may ensure that these faces agree on their intersection by building each $H_k$ one at a time and specifying its faces, as allowed by the lemma.
\end{proof}

\begin{rmk}
The special outer horn lifting property is usually stated in terms of the edge $0\to 1$ or the edge $(n-1)\to n$ being an \emph{equivalence} in our quasi-category $X$, meaning it becomes an isomorphism in the homotopy category of $X$, rather than extending to a copy of $J$. However, these two conditions are the same in quasi-categories. Given that fact, our approach here gives a more direct proof of the special outer horn lifting property of quasi-categories than in previous sources, such as the original \cite{Joyal:published}, whose arguments involve more abstract methods, transferring the special horn lifting problem to an adjoint lifting problem.
\end{rmk}

\begin{rmk}
A missing ingredient for making this result a true generalization of special outer horn lifting is a proof that an edge in a quasi-2-Segal set $X$ extends to some $J\to X$ if and only if it becomes an ``isomorphism'' in the ``homotopy 2-Segal set''. While something along these lines might be turn out to be true, we see in Appendix \ref{sub:rmkhtpy} that the idea of the ``homotopy 2-Segal set of a quasi-2-Segal set'' does not work as nicely as that of the homotopy category of a quasi-category does.
\end{rmk}

\subsection{The model structure}

Recall from \cite{Feller:generalizing} the notion of a \emph{homotopically-behaved model structure}, which is a Cisinski model structure on $\sSet$ whose fibrant objects have fillers of all $J$-augmented horn inclusions. One of our main results from that paper is a relatively simple criterion for a set of inclusions to yield such a model structure.

\begin{cor}\label{cor:otherpaper}\cite{Feller:generalizing}
Given a set of monomorphisms $S$ of simplicial sets such that the set of pushout-products $S\boxprod (\partial\Delta[1]\hookrightarrow \Delta[1])$
is contained in $\ol{S}$, there exists a homotopically-behaved model structure on $\sSet$ whose fibrant objects are those with lifts against $S$ and all $J$-augmented horn inclusions.
\end{cor}

We showed that the set $S=\TwoSegHorn$ satisfies the hypothesis of Corollary \ref{cor:otherpaper} in Corollary \ref{cor:2seganodyneclosedunderpushoutprod}, which means that we have a homotopically-behaved model structure on $\sSet$ where the fibrant objects are precisely the quasi-2-Segal sets with fillers of $J$-augmented horn inclusions. However, we saw in Proposition \ref{prop:augmentedhornlifting} that all quasi-2-Segal sets have fillers of $J$-augmented horn inclusions, so the fibrant objects of this model structure are in fact precisely the quasi-2-Segal sets. Furthermore, Proposition \ref{prop:spinesandhornsweakequivalences} implies that this model structure is the localization of the minimal model structure at the set of 2-Segal spine inclusions. We have therefore proved the following theorem.

\begin{thm}\label{thm:modelstructure}
There exists a Cisinski model structure on $\sSet$ whose fibrant objects are precisely the quasi-2-Segal sets. This model structure is the localization of the minimal model structure at the set of 2-Segal spine inclusions, and it is homotopically-behaved in the sense of \cite{Feller:generalizing}.
\end{thm}

By \cite[Prop.~E.1.10]{Joyal:theory}, a model structure is uniquely determined by its cofibrations and fibrant objects, so the model structure in Theorem \ref{thm:modelstructure} is necessarily unique. A more complete description of this model structure, including a characterization of the weak equivalences, comes from Cisinski's general theory; see \cite[Def.~2.4.18]{Cisinski:Cambridge}.

\subsection{Edgewise subdivision}

We may use our model structure to prove a quasi-2-Segal set analogue of the edgewise subdivision criterion from \cite{BOORS:edgewise}, which says that a simplicial object is 2-Segal if and only if its edgewise subdivision is Segal. Recall that there is a functor $\varepsilon\colon \Delta\to\Delta$ which sends $[n]$ to $[n]^{\op}\star [n]\cong [2n+1]$, which induces a functor $\esd\colon\sSet\to \sSet$ sending a simplicial set $X$ to its \emph{edgewise subdivision} $\esd X=X\circ \varepsilon$. The goal of this subsection is to show that a simplicial set $X$ is a quasi-2-Segal set if and only if its edgewise subdivision is a quasi-category.

We begin by unpacking what it means for the edgewise subdivision to be a quasi-category. An $n$-simplex of the edgewise subdivision $\esd X$ corresponds to a $(2n+1)$-simplex of $X$, with the $d_{\ell}$ face corresponding to the $d_{n-\ell} d_{n+1+\ell}$ face of that $(2n+1)$-simplex. A horn $\Lambda^i[n]$ in the edgewise subdivision therefore corresponds to a map $A^i[2n+1]\to X$ where $A^i[2n+1]\hookrightarrow \Delta[2n+1]$ is the union of each $d_{n-\ell} d_{n+1+\ell}$ face for $0\leq \ell <i$ and $i<\ell\leq n$. We therefore see that the edgewise subdivision being a quasi-category is equivalent to the simplicial set $X$ having lifts of each $A^i[2n+1]\to X$ for every $n\geq 2$ and $0<i<n$. We also observe that the $j\to j+1$ edge of an $n$-simplex in the edgewise subdivision corresponds to the 3-simplex $(n-j-1) \to (n-j)\to (n+j+1)\to (n+j+2)$ of $\Delta[2n+1]$, so the spine of the $n$-simplex corresponds to the union $I[2n+1]$ of those 3-simplices for $0\leq j<n$. With these preliminary observations, we are ready to prove the edgewise subdivision criterion.

\begin{prop}[Edgewise subdivision criterion]\label{prop:edgewise}
A simplicial set $X$ is a quasi-2-Segal set if and only if its edgewise subdivision is a quasi-category.
\end{prop}

\begin{proof}
We begin by showing the forward implication. Because the quasi-2-Segal sets are the fibrant objects in our model structure, they have lifts of all inclusions which are weak equivalences. Therefore, to show that the edgewise subdivision of every quasi-2-Segal set is a quasi-category, it suffices to show that each $A^i[2n+1]\hookrightarrow \Delta[2n+1]$ is a weak equivalence in our model structure. However, the arguments in \cite[Lem.~3.5]{JT} apply just as readily to show that every $A^i[2n+1]\hookrightarrow \Delta[2n+1]$ (corresponding to an inner horn inclusion) is a weak equivalence if every $I[2n+1]\hookrightarrow \Delta[2n+1]$ is a weak equivalence. But restricting $I[2n+1]$ along $\Lambda^{0,2}[3]\hookrightarrow \Delta[3]$ for each non-degenerate 3-simplex of $I[2n+1]$ yields a triangulation $\mcT$ of the $(2n+2)$-gon, and therefore a weak equivalence $\mcT\hookrightarrow I[2n+1]$. Because the inclusion $\mcT\hookrightarrow \Delta[2n+1]$ is a weak equivalence, so is $I[2n+1]\hookrightarrow \Delta[2n+1]$ by the 2-out-of-3 property. We have therefore shown that the edgewise subdivision of a quasi-2-Segal set is necessarily a quasi-category.

We now show the reverse implication. As discussed above, the edgewise subdivision of $X$ being a quasi-category is equivalent to $X$ having fillers of the inclusions $A^i[2n+1]\hookrightarrow \Delta[2n+1]$. To show that this condition implies $X$ is a quasi-2-Segal set, we show that every 2-Segal horn inclusion of the form $\Lambda^{0,j}[k]\hookrightarrow \Delta[k]$ is a retract of $A^{j-1}[2k-1]\hookrightarrow \Delta[2k-1]$. This fact, together with its dual, implies that $X$ is a quasi-2-Segal set by Theorem \ref{thm:pathspacecriterion}.

Given $k\geq 3$ and $2\leq j\leq k$, let us show that $\Lambda^{0,j}[k]\hookrightarrow \Delta[k]$ is a retract of $A^{j-1}[2k-1]\hookrightarrow \Delta[2k-1]$. First, we define the retract
\[\begin{tikzcd}
	{\Delta[k]} & {\Delta[2k-1]} & {\Delta[k]\nospace{.}}
	\arrow["f"', hook, from=1-1, to=1-2]
	\arrow["p"', from=1-2, to=1-3]
\end{tikzcd}\]
where $f$ maps $0\mapsto k-j$ and $\ell\mapsto k-1+\ell$ for $\ell>0$, and where $p$ maps $\ell\mapsto 0$ for $\ell<k$ and $\ell\mapsto \ell-k+1$ for $\ell\geq k$. We then have
\[
f(d_{\ell}\Delta[k])\cin d_{k-\ell}d_{k+\ell-1}\cin A^{j-1}[2k-1]
\]
if $0<\ell <j$ or $j<\ell\leq k$, as well as
\[
p(d_{k-1+\ell}\Delta[2k-1])\cin d_{\ell} \Delta[k]
\]
for each $1<\ell\leq k$, which altogether says we have our desired retract diagram



\[\begin{tikzcd}
	{\Lambda^{0,j}[k]} & {A^{j-1}[2k-1]} & {\Lambda^{0,j}[k]} \\
	{\Delta[k]} & {\Delta[2k-1]} & {\Delta[k]\nospace{.}}
	\arrow["f"', hook, from=2-1, to=2-2]
	\arrow["p"', from=2-2, to=2-3]
	\arrow[hook, from=1-1, to=2-1]
	\arrow[hook, from=1-3, to=2-3]
	\arrow[hook, from=1-2, to=2-2]
	\arrow[hook, from=1-1, to=1-2]
	\arrow[from=1-2, to=1-3]
\end{tikzcd}\]
We have therefore shown that a simplicial set is a quasi-2-Segal set if its edgewise subdivision is a quasi-category.
\end{proof}

\begin{rmk}
We expect a stronger statement to hold, that the set of inclusions $\{A^i[2n+1]\hookrightarrow \Delta[2n+1]\}$ actually generates the class of 2-Segal anodyne inclusions, which would yield a direct proof of the edgewise subdivision criterion without appealing to a model structure. Our proof of Proposition \ref{prop:edgewise} shows one of the necessary containments of this stronger statement, but we do not show that each $A^i[2n+1]\hookrightarrow \Delta[2n+1]$ is 2-Segal anodyne, only that each is a weak equivalence in the quasi-2-Segal model structure.
\end{rmk}

\section{Quasi-2-Segal sets vs 2-Segal spaces}\label{sec:vs}

A vital aspect of $(\infty,1)$-category theory is the pair of adjunctions between simplicial sets and simplicial spaces providing Quillen equivalences between the Joyal model structure and Rezk's model structure for complete Segal spaces, first proved in \cite{JT}. (The idea of a Quillen equivalence is that it provides the appropriate notion of equivalence for model categories; see \cite[\S 1.3]{Hovey} for an explicit definition.) Since each model of $(\infty,1)$-category has its benefits in different situations, it can be extremely helpful to be able to move freely between them. Our goal in this section is to prove that the same adjunctions provide Quillen equivalences between our quasi-2-Segal set model structure and a model structure for complete 2-Segal spaces. This result follows with relatively little extra work from very general results of Ara \cite{Ara}, building on Cisinski's work in \cite{Cisinski:Asterisque}, which gives an explicit description of a model structure on $\ssSet$ that is Quillen equivalent to our quasi-2-Segal set model structure. What we show here is that the model structure from Ara's framework is precisely the same as the model structure one gets by naively generalizing the model structure for complete Segal spaces due to Rezk \cite{Rezk:CSS}.

Recall the adjunctions



\[\begin{tikzcd}
	\sSet && \ssSet && \sSet
	\arrow[""{name=0, anchor=center, inner sep=0}, "{p_1^{\ast}}", shift left=2, from=1-1, to=1-3]
	\arrow[""{name=1, anchor=center, inner sep=0}, "{i_1^{\ast}}", shift left=2, from=1-3, to=1-1]
	\arrow[""{name=2, anchor=center, inner sep=0}, "{t_!}", shift left=2, from=1-3, to=1-5]
	\arrow[""{name=3, anchor=center, inner sep=0}, "{t^!}", shift left=2, from=1-5, to=1-3]
	\arrow["\dashv"{anchor=center, rotate=-90}, draw=none, from=2, to=3]
	\arrow["\dashv"{anchor=center, rotate=-90}, draw=none, from=0, to=1]
\end{tikzcd}\]
from \cite{JT}, where $p_1^{\ast}$ is the functor sending a simplicial set $X$ its corresponding discrete simplicial space, and where the functors $t^!$ and $t_!$ are as described in the paragraph before \cite[Lem.~2.11]{JT}. Viewing a simplicial space $W$ as a grid of sets where the $n$th column is the simplicial set $W_n$, the functor $p_1^{\ast}$ gives a ``horizontal embedding'' of $\sSet$ into $\ssSet$.

One of the characterizations of completeness for Segal spaces is being local with respect to the inclusion $p_1^{\ast}(\{0\}\hookrightarrow J)$ \cite[Prop.~6.4]{Rezk:CSS}. We take this as our definition for completeness of 2-Segal spaces.

\begin{defn}\label{def:completeness}
We say that a vertical Reedy fibrant 2-Segal space $W$ is \emph{complete} if the map
\[
\Map(p_1^{\ast}(J),W)\to \Map(p_1^{\ast}\Delta[0],W) \cong W_0
\]
induced by $p_1^{\ast}(\{0\}\hookrightarrow J)$ is a Kan-Quillen weak equivalence.
\end{defn}

We therefore get a Cisinski model structure on $\ssSet$ whose fibrant objects are complete 2-Segal spaces by localizing the 2-Segal model structure from \cite{DK} at the map $p_1^{\ast}(\{0\}\hookrightarrow J)$. We can now state the main theorem of this section.

\begin{thm}\label{thm:quillenequiv}
The adjunctions $p_1^{\ast}\dashv i_1^{\ast}$ and $t_!\dashv t^!$ are Quillen equivalences between the model structure for quasi-2-Segal sets and the model structure for complete 2-Segal spaces.
\end{thm}

\begin{warn}\label{warn:completeness}
Other notions of completeness for 2-Segal spaces have been defined in the literature---Dyckerhoff-Kapranov offer a definition in \cite[Rmk.~9.3.13]{DK}, and G\'alvez-Kock-Tonks also have a different definition in \cite[Def.~2.1]{GKT:partII}---but neither of these is a generalization of completeness for Segal spaces in the spirit of our definition.

However, G\'alvez-Kock-Tonks also refer to a condition they call Rezk completeness \cite[5.13]{GKT:partII}, which is in the same spirit as our definition of completeness, but is a direct generalization of the definition in \cite[\S 6]{Rezk:CSS}. This definition of Rezk completeness for 2-Segal spaces is also used in \cite{HK:culf} and \cite{HK:free}. (Recall that these sources refer to 2-Segal spaces as decomposition spaces.) For a general 2-Segal space their definition of Rezk completeness is stronger than Definition \ref{def:completeness}, as we discuss below in Subsection \ref{sub:equivalences}.
\end{warn}

The goal of the remainder of this section is to prove Theorem \ref{thm:quillenequiv}. We begin by explaining the model structure on simplicial spaces we get from Ara's theory, and then showing this model structure is precisely the complete 2-Segal space model structure.

Let us state Ara's theorem in our particular setting.

\begin{prop}\label{prop:ara}
The adjunctions $p_1^{\ast}\dashv i_1^{\ast}$ and $t_!\dashv t^!$ are Quillen equivalences between the model structure for quasi-2-Segal sets and the localization of the vertical Reedy model structure on $\ssSet$ at the class of maps
\[
p_1^{\ast}(S)\cup \{p_1^{\ast}((\{0\}\times Y)\hookrightarrow (J \times Y))\mid Y\in \sSet\},
\]
where $S$ is the set of 2-Segal spine inclusions, or, equivalently, where $S=\TwoSegHorn$.
\end{prop}

\begin{proof}
Apply \cite[Thm.~4.11]{Ara} to the case where $A=\Delta$ and where the model structure is our model structure for quasi-2-Segal sets. The description of the localizing maps is \cite[4.1]{Ara}. Either choice of the set $S$ works because the model structure for quasi-2-Segal sets is the localization of the minimal model structure either at $\TwoSegHorn$ or at the set of 2-Segal spine inclusions by Proposition \ref{prop:spinesandhornsweakequivalences}.
\end{proof}

In other words, the model structure from Ara's theory is the localization of the model structure for complete 2-Segal spaces at the maps
\[
\{p_1^{\ast}((\{0\}\times Y)\hookrightarrow (J \times Y))\mid Y\in \sSet\}.
\]
Our goal is therefore to show that these model structures are the same by showing that these maps must necessarily be weak equivalences in the model structure for complete 2-Segal spaces. Because each of the maps in
\[
p_1^{\ast}(\TwoSegHorn)\cup \{p_1^{\ast}(\{0\}\hookrightarrow J)\}
\]
is a monomorphism and hence a trivial cofibration, it suffices to show that each $p_1^{\ast}((\{0\}\times Y)\hookrightarrow (J \times Y))$ is in the saturated class generated by $p_1^{\ast}(\TwoSegHorn)\cup \{p_1^{\ast}(\{0\}\hookrightarrow J)\}$. However, because $p_1^{\ast}$ is a left adjoint and so preserves colimits, it suffices to show that each $(\{0\}\times Y)\hookrightarrow (J \times Y)$ is in the saturated class generated by $\TwoSegHorn\cup\{\{0\}\hookrightarrow J)\}$ in $\sSet$. But each $(\{0\}\times Y)\hookrightarrow (J \times Y)$ is the pushout-product $(\varnothing\hookrightarrow Y)\boxprod (\{0\}\hookrightarrow J)$, and we showed in \cite{Feller:minimal} that these pushout-products are in the saturated class generated by a set of certain inclusions we call iso-horn inclusions. It therefore suffices to show that these iso-horn inclusions are each in $\ol{\TwoSegHorn\cup \{\{0\}\hookrightarrow J\}}$. Let us give a definition of iso-horn inclusions.

\begin{defn}\label{def:isohorns}
Fix $n\geq 1$ and $0\leq i\leq n-1$, denote by $\tri_i[n]$ the nerve of the category $[n]_i$



\[\begin{tikzcd}[column sep=small]
	0 & \ldots & {i-1} & i & {i+1} & {i+2} & \ldots & n\nospace{.}
	\arrow[from=1-1, to=1-2]
	\arrow[from=1-2, to=1-3]
	\arrow[from=1-3, to=1-4]
	\arrow[shift right=1, from=1-4, to=1-5]
	\arrow[shift right=1, from=1-5, to=1-4]
	\arrow[from=1-5, to=1-6]
	\arrow[from=1-6, to=1-7]
	\arrow[from=1-7, to=1-8]
\end{tikzcd}\]
We call $\tri_i[n]$ an \emph{isoplex}. For $0\leq j\leq n$, we let the $j$th face of the isoplex $\tri_i[n]$, denoted $d_j \tri_i[n]$, be the full subcomplex on all but the $j$th vertex. Let $\bV_i[n]$ be the union of all of the faces $d_j \tri_i[n]$ except for $j=i$. We call $\bV_i[n]$ an \emph{iso-horn}, and the inclusion $\bV_i[n]\hookrightarrow \tri_i[n]$ an \emph{iso-horn inclusion}. Denote by $\IsoHorn$ the set of all iso-horn inclusions.
\end{defn}

Observe that when $j=i$ or $i+1$, the $j$th face of $\tri_i[n]$ is an $(n-1)$-simplex $\Delta[n-1]$, and otherwise it is an $(n-1)$-isoplex. We can view an isoplex as an ``isomorphism  of $(n-1)$-simplices'' from the $i+1$ face to the $i$th face. Similarly, we can view an iso-horn $\bV_i[n]$ as an $(n-1)$-simplex extended by an isomorphism along its boundary.

We get the following proposition as a consequence of our results in \cite{Feller:minimal}.

\begin{prop}\label{prop:pushoutsinisohornsaturation}
For every simplicial set $Y$, the inclusion $(\{0\}\times Y)\hookrightarrow (J \times Y)$ is in $\ol{\IsoHorn}$.
\end{prop}

\begin{proof}
Since the inclusion $(\{0\}\times Y)\hookrightarrow (J \times Y)$ is the pushout-product $(\varnothing\hookrightarrow Y)\boxprod (\{0\}\hookrightarrow J)$, we may use the fact that the class of pushout-products $(X\hookrightarrow Y)\boxprod (\{0\}\hookrightarrow J)$ for all inclusions $X\hookrightarrow Y$ generates the same saturated class as the set of iso-horn inclusions, as we showed in \cite{Feller:minimal}.
\end{proof}

By this proposition, to show that each inclusion $(\{0\}\times Y)\hookrightarrow (J \times Y)$ is in the saturated class generated by $\TwoSegHorn\cup \{\{0\}\hookrightarrow J\}$, it is enough to prove that the set $\IsoHorn$ is contained in $\ol{\TwoSegHorn\cup \{\{0\}\hookrightarrow J\}}$. Furthermore, because $\{0\}\hookrightarrow J$ is precisely the iso-horn $\bV_0[1]\hookrightarrow \tri_0[1]$, it suffices to prove the following lemma.

\begin{lem}\label{lem:isohornpushoutsof2seg}
For each $n\geq 2$ and $0\leq i<n$, the iso-horn inclusion $\bV_i[n]\hookrightarrow \tri_i[n]$ is in $\ol{\TwoSegHorn}$.
\end{lem}

\begin{proof}
Let $B_{\ell}$ be the unique $(n+2\ell-1)$-simplex of $\tri_i[n]$ traced by the path
\[
0\to 1\to \ldots\to i\to (i+1) \to i \to (i+1) \to \ldots \to i \to (i+1)\to i \to (i+2) \to \ldots \to n,
\]
which passes through the $i+1$ vertex $\ell$ times. The $(n-1)$-simplex $B_0$ is simply $d_{i+1}\tri_i[n]$ which is in $\bV_i[n]$. Every $k$-simplex of $\tri_i[n]$ is contained in $B_n$ for some $n\geq 0$, so we can decompose our iso-horn inclusion as
\[
\bV_i[n] = \bV_i[n]\cup B_0 \hookrightarrow \bV_i[n]\cup B_1 \hookrightarrow \ldots \hookrightarrow \bV_i[n]\cup B_i \hookrightarrow \ldots \tri_i[n].
\]
But each inclusion $\bV_i[n]\cup B_{\ell-1} \hookrightarrow \bV_i[n]\cup B_{\ell}$ is a pushout of a 2-Segal horn inclusion $\Lambda^{i,i+2\ell}[n+2\ell-1]\hookrightarrow \Delta[n+2\ell-1]$. To see why, observe that each face $d_j B_{\ell}$ for $j<i$ or $j>i+2\ell$ is contained in $\bV_i[n+2\ell-1]$, while for $i<j<i+2\ell$ the face $d_j B_{\ell}$ is a degeneracy of a simplex contained in $B_{\ell-1}$. Meanwhile, the face $d_i d_{i+2\ell} B_{\ell}$ is not contained in $B_{\ell-1}$ because it is non-degenerate yet passes through the $i+1$ vertex $\ell$ times, and it is not contained in $\bV_i[n]$ because it contains the $d_i$ face of $\tri_i[n]$; the faces $d_i B_{\ell}$ and $d_{i+2\ell} B_{\ell}$ are therefore not in $\bV_i[n]\cup B_{\ell}$ either. We have shown that our iso-horn inclusion is a countable composite of pushouts of 2-Segal horn inclusions, proving the lemma.
\end{proof}

We have now proved Theorem \ref{thm:quillenequiv}. Let us summarize the argument.

\begin{proof}[Proof of Theorem \ref{thm:quillenequiv}]
By Proposition \ref{prop:pushoutsinisohornsaturation}, the inclusions $(\{0\}\times Y)\hookrightarrow (J\times Y)$ are in $\ol{\IsoHorn}$, and by Lemma \ref{lem:isohornpushoutsof2seg} we have $\ol{\IsoHorn}\cin \ol{\TwoSegHorn\cup \{\{0\}\hookrightarrow J\}}$. Because the left adjoint $p_1^{\ast}$ preserves colimits, we therefore have that each inclusion
\[
p_1^{\ast}\left((\{0\}\times Y)\hookrightarrow (J\times Y)\right)
\]
is a trivial cofibration, and hence a weak equivalence in the model structure for complete 2-Segal spaces because it is contained in the saturated class generated by the set of inclusions $p_1^{\ast}(\TwoSegHorn)\cup \{p_1^{\ast}(\{0\}\hookrightarrow J)\}$ which are all trivial cofibrations. The model structure for complete 2-Segal spaces is therefore precisely the model structure from Proposition \ref{prop:ara}.
\end{proof}

We conclude by showing that the path space criterion also applies to complete 2-Segal spaces.

\begin{prop}\label{prop:pathspacecomplete}
Given a vertical Reedy fibrant 2-Segal space $W$, the following are equivalent.
\begin{enumerate}
    \item The 2-Segal space $W$ is complete.
    \item One of the path spaces of $W$ is complete.
    \item Both of the path spaces of $W$ are complete.
\end{enumerate}
\end{prop}

\begin{proof}
By symmetry, it suffices to show that a 2-Segal space is complete if and only if its left path space $P^{\triangleleft}(X)$ is complete. The completeness condition for the path space says that
\[
\Map(p_1^{\ast}(J),P^{\triangleleft}W)\to (P^{\triangleleft}W)_0
\]
is a Kan-Quillen weak equivalence, but this map is the map
\[
\Map(p_1^{\ast}(\tri_1[2]),W)\to W_1
\]
induced by $d_2\colon \Delta[1] \hookrightarrow \tri_1[2]$ (which is the join of $\{0\}\hookrightarrow J$ with a point). Hence, we see that the path space $P^{\triangleleft}W$ being local with respect to $\{0\}\hookrightarrow J$ is equivalent to $W$ itself being local with respect to $d_2\colon \Delta[1] \hookrightarrow \tri_1[2]$.

To show that the path space $P^{\triangleleft}W$ is complete if $W$ is, it suffices to show that the inclusion $p_1^{\ast}(d_2\colon \Delta[1] \hookrightarrow \tri_1[2])$ is a weak equivalence in the complete 2-Segal space model structure. Again, because $p_1^{\ast}$ preserves colimits, it suffices to observe that $d_2\colon \Delta[1] \hookrightarrow \tri_1[2]$ is a composite of pushouts of iso-horn inclusions
\[
\Delta[1]\hookrightarrow \bV_1[2]\hookrightarrow \tri_1[2],
\]
where the first inclusion is a pushout of $\bV_0[1]\hookrightarrow \tri_0[1]$, i.e., $\{0\}\hookrightarrow J$.

To show the converse, we observe that $\{0\}\hookrightarrow J$ is a retract of $d_2\colon \Delta[1] \hookrightarrow \tri_1[2]$, inducing the retract diagram



\[\begin{tikzcd}
	{\Map (p_1^{\ast}(J),W)} & {\Map (p_1^{\ast}(\tri_1[2]),W)} & {\Map (p_1^{\ast}(J),W)} \\
	{W_0} & {W_1} & {W_0\nospace{,}}
	\arrow[hook, from=1-1, to=1-2]
	\arrow[from=1-2, to=1-3]
	\arrow[from=1-2, to=2-2]
	\arrow[from=1-3, to=2-3]
	\arrow[from=1-1, to=2-1]
	\arrow[hook, from=2-1, to=2-2]
	\arrow[from=2-2, to=2-3]
\end{tikzcd}\]
where the center map is a Kan-Quillen weak equivalence if $P^{\triangleleft}W$ is complete, hence the outer vertical map is also a weak equivalence.
\end{proof}

\begin{cor}\label{cor:sdotgivescomplete}
Given an exact quasi-category $\mcC$, the simplicial space $S_{\bullet}(\mcC)$ as in \cite[Def.~7.3.1]{DK} is a complete 2-Segal space.
\end{cor}

\begin{proof}
The proof of \cite[Thm.~7.3.3]{DK} shows that the path spaces of $S_{\bullet}(\mcC)$ are in fact complete Segal spaces.
\end{proof}

\subsection{Equivalences inside quasi-2-Segal sets and 2-Segal spaces}\label{sub:equivalences}

Let us expand upon Warning \ref{warn:completeness} and discuss the different ways to define completeness for 2-Segal spaces. At the heart of the issue is generalizing the notion of \emph{equivalence} for quasi-categories and Segal spaces to the 2-Segal setting.

\begin{defn}
Given an edge in a simplicial set (or simplicial space) $f\colon\Delta[1]\to X$, let us say that $f$ is \emph{bi-invertible} if it has a left inverse and a right inverse. That is, there exist 2-simplices $\sigma$ and $\sigma'$ such that $d_1 \sigma$ and $d_1\sigma'$ are degenerate and $d_2 \sigma = d_0 \sigma'=f$. Let us say that $f$ is a \emph{true equivalence} if $f\colon \Delta[1]\to X$ factors through $\Delta[1]\hookrightarrow J$.
\end{defn}

In general, all true equivalences are bi-invertible, but being a true equivalence is much stronger, since it says that there is a single edge which is both a left and a right inverse in a coherent way. However, if $X$ is a quasi-category (or a Segal space), the converse holds, and both notions agree with the usual notion of ``equivalence.'' The same does not hold in the 2-Segal world, as the following example demonstrates.

\begin{ex}
Let $\Sp$ be the stable quasi-category of spectra, and let $W$ be the simplicial set $S^{\operatorname{set}}_{\bullet}(\Sp)$ or the simplicial space $S_{\bullet}(\Sp)$. Given a spectrum $Y$, which corresponds to an edge in $W$, we see in the following diagrams that suspension and loops give us a left and a right inverse of $Y$.




\[\begin{tikzcd}
	\ast & {\Omega Y} & \ast && \ast & Y & \ast \\
	& \ast & Y &&& \ast & {\Sigma Y} \\
	&& \ast &&&& \ast
	\arrow[from=1-1, to=1-2]
	\arrow[from=1-2, to=2-2]
	\arrow[from=1-2, to=1-3]
	\arrow[from=1-3, to=2-3]
	\arrow[from=2-2, to=2-3]
	\arrow[from=2-3, to=3-3]
	\arrow[from=1-5, to=1-6]
	\arrow[from=1-6, to=2-6]
	\arrow[from=2-6, to=2-7]
	\arrow[from=1-6, to=1-7]
	\arrow[from=1-7, to=2-7]
	\arrow[from=2-7, to=3-7]
\end{tikzcd}\]

Thus, every edge of $W$ is bi-invertible. Meanwhile, the higher coherence data in a map $J\to W$ implies that only contractible spectra are true equivalences in $W$.
\end{ex}

This example also shows that our definition of completeness is strictly weaker than the other sensible notion of Rezk completeness used in other sources. Recall that for Segal spaces there are two equivalent characterizations of completeness: the first is that a Segal space $W$ is local with respect to $\ast\to J$, as in Definition \ref{def:completeness}, while the second (and actually what Rezk takes as the definition) is that $s_0\colon W_0\to W_{\operatorname{equiv}}$ is a weak equivalence of Kan complexes, where $W_{\operatorname{equiv}}\cin W_1$ is the full sub-Kan complex on the set of equivalences. One can show that a 2-Segal space is complete in the first sense if and only if $s_0\colon W_0\to W_{\operatorname{true equiv}}$ is a weak equivalence of Kan complexes, where $W_{\operatorname{true equiv}}\cin W_1$ is the full sub-Kan complex on the set of true equivalences. However, as the above example shows, the subspace of bi-invertible edges need not be equivalent to the subspace of true equivalences. Demanding instead that $s_0\colon W_0\to W_{\operatorname{bi-inv}}$ be a weak equivalence of Kan complexes, as in the definition of Rezk completeness used in \cite{GKT:partII}, \cite{HK:culf}, and \cite{HK:free}, yields a stronger condition. In particular, we see via Corollary \ref{cor:sdotgivescomplete} that $S_{\bullet}(\Sp)$ is complete in the sense of Definition \ref{def:completeness} but is not Rezk complete in the sense of \cite{GKT:partII}, \cite{HK:culf}, and \cite{HK:free}.

\appendix

\section{Additional Remarks}

\subsection{A remark about ``homotopy 2-Segal sets''}\label{sub:rmkhtpy}

Recall that, by identifying categories with their nerves, we can view $\tau_1$, the left adjoint of the nerve functor, as a functor $\sSet\to \sSet$. The simplicial set $\tau_1 X$ has the universal property that any map from a simplicial set $X$ to the nerve of a category factors uniquely through the unit map $X\to \tau_1 X$. When $X$ is a quasi-category, we call $\tau_1 X$ the \emph{homotopy category of $X$}. An edge in a quasi-category $\Delta[1]\to X$ is called an \emph{equivalence} if $\Delta[1]\to X\to \tau_1 X$ extends to a map $J\to \tau_1 X$, i.e., if the edge becomes an isomorphism in the homotopy category.

One reasonable definition of equivalence for quasi-2-Segal sets would be analogous: let $\tau_2$ denote the left adjoint to the inclusion of the full subcategory of 2-Segal sets into $\sSet$, then an edge in a quasi-2-Segal set $X$ is an \emph{equivalence} if its image in $X\to \tau_2 X$ extends to a copy of $J\to \tau_2 X$. However, we have avoided this discussion because $\tau_2$ is much harder to describe explicitly than $\tau_1$ and we do not expect $\tau_2$ to play as important of a role in quasi-2-Segal theory as $\tau_1$ does for quasi-categories.

One of the ways in which $\tau_2$ is not as nicely behaved as $\tau_1$ is that if $X$ is a quasi-2-Segal set, then $\tau_2 X$ is not quite the ``homotopy 2-Segal set'' of $X$ in the way that one might hope. In particular, if we view an $(n+1)$-simplex $H$ with degenerate edge $i\to i+1$ to be a homotopy from its $d_{i+1}$ face to its $d_i$ face, then we view two $n$-simplices $\sigma$ and $\sigma'$ of a simplicial set $X$ as \emph{homotopic} to one another, denoted by $\sigma\sim\sigma'$, if there exists a zig-zag of homotopies connecting them. When $X$ is a quasi-category, one can show that the $n$-simplices of $\tau_1 X$ correspond to equivalence classes of $n$-simplices in $X$ under the relation $\sim$. However, the same cannot necessarily be said of $\tau_2 X$ for a quasi-2-Segal set $X$. It is true that homotopic simplices of $X$ are identified in $\tau_2 X$ because 2-Segal sets do not have nontrivial homotopies by \cite{FGKPW}, but the converse is not true.

\begin{ex}
We find a counterexample in $S_{\bullet}^{\operatorname{set}}(\Ab)$, where one can show that two $n$-simplices are homotopic in $S_{\bullet}^{\operatorname{set}}(\Ab)$ if and only if they are isomorphic as diagrams in $\Ab$. Consider the 3-simplices

\[
\adjustbox{scale=0.9}{
\begin{tikzcd}[column sep=small]
	0 & {\mathbb{Z}/2\mathbb{Z}} & {\mathbb{Z}/4\mathbb{Z}} & {\mathbb{Z}/4\mathbb{Z}\oplus\mathbb{Z}/2\mathbb{Z}} & 0 & {\mathbb{Z}/2\mathbb{Z}} & {\mathbb{Z}/2\mathbb{Z}\oplus \mathbb{Z}/2\mathbb{Z}} & {\mathbb{Z}/4\mathbb{Z}\oplus\mathbb{Z}/2\mathbb{Z}} \\
	& 0 & {\mathbb{Z}/2\mathbb{Z}} & {\mathbb{Z}/2\mathbb{Z}\oplus\mathbb{Z}/2\mathbb{Z}} && 0 & {\mathbb{Z}/2\mathbb{Z}} & {\mathbb{Z}/2\mathbb{Z}\oplus\mathbb{Z}/2\mathbb{Z}} \\
	&& 0 & {\mathbb{Z}/2\mathbb{Z}} &&& 0 & {\mathbb{Z}/2\mathbb{Z}} \\
	&&& 0 &&&& 0
	\arrow[hook, from=2-2, to=2-3]
	\arrow[hook, from=1-1, to=1-2]
	\arrow[hook, from=3-3, to=3-4]
	\arrow[two heads, from=3-4, to=4-4]
	\arrow["{\pi_1}", two heads, from=2-4, to=3-4]
	\arrow["{1\oplus 1}", two heads, from=1-4, to=2-4]
	\arrow[two heads, from=2-3, to=3-3]
	\arrow["{i_0}"', hook, from=2-3, to=2-4]
	\arrow[two heads, from=1-2, to=2-2]
	\arrow["{1\oplus 1}", two heads, from=1-8, to=2-8]
	\arrow["{i_1}"', hook, from=2-7, to=2-8]
	\arrow["{\pi_0}", two heads, from=2-8, to=3-8]
	\arrow[hook, from=2-6, to=2-7]
	\arrow[hook, from=3-7, to=3-8]
	\arrow[hook, from=1-5, to=1-6]
	\arrow[two heads, from=1-6, to=2-6]
	\arrow[two heads, from=2-7, to=3-7]
	\arrow[two heads, from=3-8, to=4-8]
	\arrow["{\pi_1}", two heads, from=1-7, to=2-7]
	\arrow["{2\oplus 1}", hook, from=1-7, to=1-8]
	\arrow["{i_0}", hook, from=1-6, to=1-7]
	\arrow[two heads, from=1-3, to=2-3]
	\arrow["{i_0}", hook, from=1-3, to=1-4]
	\arrow[hook, from=1-2, to=1-3]
\end{tikzcd}
}
\]
of $S_{\bullet}^{\operatorname{set}}(\Ab)$. The $d_0$ and $d_2$ faces of these simplices are respectively isomorphic, and so are respectively identified in $\tau_2 S_{\bullet}^{\operatorname{set}}(\Ab)$, which means that these 3-simplices are also identified by uniqueness of 2-Segal spine extensions. However, these 3-simplices are not homotopic in $S_{\bullet}^{\operatorname{set}}(\Ab)$ because they are not isomorphic as diagrams in $\Ab$.
\end{ex}

\subsection{A remark about higher Segal conditions}\label{sec:highersegal}

In \cite{DK}, Dyckerhoff-Kapranov also introduced $d$-Segal spaces for $d>2$, defined in terms of $d$-dimensional cyclic polytopes. For each $d\geq 2$, there is an ``upper'' and a ``lower'' $d$-Segal condition, and we say that a simplicial object is $d$-Segal if it satisfies both the upper and lower $d$-Segal conditions. We refer to \cite{Poguntke} for explicit definitions. For our present purposes, it suffices to note the following facts:

\begin{itemize}
    \item \cite[Prop.~2.10]{Poguntke} If a simplicial object $X$ is upper or lower $n$-Segal, then it is also (fully) $d$-Segal for all $d> n$. In particular, if a simplicial set is lower or upper 2-Segal then it is $d$-Segal for all $d>2$.\\
    \item \cite[Prop.~2.7]{Poguntke} A simplicial object $X$ is lower (resp.~upper) 2-Segal if and only if its left (resp.~right) path space is Segal.
\end{itemize}

One might wonder whether there exists a good notion of quasi-$d$-Segal space for all $d\geq 1$. Unfortunately, the situation is not as nice for $d>2$ because there cannot be a Cisinski model structure whose fibrant objects are precisely the quasi-$d$-Segal sets. The reason is that there exist 3-Segal simplicial sets (and hence $d$-Segal for all $d>3$ too) which are not fibrant in any Cisinski model structure, as in the following example.

\begin{ex}
The iso-horn $\bV_1[2]$ is lower 2-Segal, and hence $d$-Segal for all $d\geq 3$ by the first fact above. To see why, by the second fact above it suffices to observe that the left path space $P^{\triangleleft}(\bV_1[2])$ is Segal, i.e., the nerve of a category. This path space decomposes as $P^{\triangleleft}(\bV_1[2])\cong \Delta[1]\sqcup P^{\triangleleft}(J)$, which is the disjoint union of nerves of categories. (Since $J$ is the nerve of a category, so is its path space.)

As explained in \cite{Feller:minimal}, the fibrant objects in a Cisinski model structure on $\sSet$ must have fillers of all iso-horn inclusions, so the iso-horn $\bV_1[2]$ itself cannot be fibrant in any such model structure.
\end{ex}

It could still be possible that for $d>2$ there are model structures whose fibrant objects are precisely the quasi-$d$-Segal sets if we work with a different class of cofibrations, but it would likely yield a theory which looks quite different from that of quasi-categories and quasi-2-Segal sets.


\begin{thebibliography}{8}

    \bibitem{Ara}
{\sc D.~Ara}.
\newblock {\em Higher quasi-categories vs higher {R}ezk spaces}.
\newblock Journal of $K$-Theory {\bf 14} (2014), no.~3, pp.~701--749.

    \bibitem{BOORS:objects}
{\sc J.~Bergner, A.~Osorno, V.~Ozornova, M.~Rovelli, {\rm and }C.~Scheimbauer}.
\newblock {\em 2-{S}egal objects and the Waldhausen construction}.
\newblock Alg.~Geom.~Topol.~{\bf 21} (2021), no.~3, pp.~1267--1326.

    \bibitem{BOORS:edgewise}
{\sc J.~Bergner, A.~Osorno, V.~Ozornova, M.~Rovelli, {\rm and }C.~Scheimbauer}.
\newblock {\em The edgewise subdivision criterion for 2-{S}egal objects}.
\newblock Proc.~Amer.~Math.~Soc.~{\bf 148} (2020), pp.~71--82.

    \bibitem{Carlier}
{\sc L.~Carlier}.
\newblock {\em Incidence bicomodules, {M}\"obius inversion, and a {R}ota formula for infinity adjunctions}.
\newblock Algebr.~Geom.~Topol.~{\bf 20} (2020), pp.~169--213.

    \bibitem{Cisinski:Asterisque}
{\sc D.-C.~Cisinski}.
\newblock {\em Les pr\'{e}faisceaux comme type d'homotopie}, vol.~308 of Ast\'{e}risque,
\newblock Soc.~Math.~France, 2006.

    \bibitem{Cisinski:Cambridge}
{\sc D.-C.~Cisinski}.
\newblock {\em Higher Categories and Homotopical Algebra}, 
\newblock Cambridge Studies in Advanced Mathematics {\bf 180}, Cambridge University Press, 2019.

    \bibitem{DK}
{\sc T.~Dyckerhoff {\rm and }M.~Kapranov}.
\newblock {\em Higher {S}egal spaces}.
\newblock No.~2244 in Lecture Notes in Mathematics, Spinger-Verlag, 2019.

    \bibitem{Feller:generalizing}
{\sc M.~Feller}.
\newblock {\em Generalizing quasi-categories via model structures on simplicial sets}.
\newblock Preprint. \arxiv{2111.06512}

    \bibitem{Feller:minimal}
{\sc M.~Feller}
\newblock {\em A horn-like characterization of the fibrant objects in the minimal model structure on simplicial sets}.
\newblock Preprint. \arxiv{2201.13400}

    \bibitem{FGKPW}
{\sc M.~Feller, R.~Garner, J.~Kock, M.~U.~Proulx, {\rm and }M.~Weber}.
\newblock {\em Every 2-{S}egal space is unital}.
\newblock Commun.~Contemp.~Math.~{\bf 23} (2021), 2050055.

    \bibitem{GKT:partI}
{\sc I.~G{\'a}lvez-Carrillo, J.~Kock, {\rm and }A.~Tonks}.
\newblock {\em Decomposition spaces, incidence algebras and {M}{\"o}bius
  inversion {I}: basic theory}.
\newblock Adv.~Math.~{\bf 331} (2018), pp.~952--1015.

    \bibitem{GKT:partII}
{\sc I.~G{\'a}lvez-Carrillo, J.~Kock, {\rm and }A.~Tonks}.
\newblock {\em Decomposition spaces, incidence algebras and {M}{\"o}bius inversion {II}: completeness, length filtration, and finiteness}.
\newblock Adv.~Math.~{\bf 333} (2018), pp.~1242--1292.

    \bibitem{HK:culf}
{\sc P.~Hackney {\rm and }J.~Kock}.
\newblock {\em Culf maps and edgewise subdivision}.
\newblock Preprint. \arxiv{2210.11191}.

    \bibitem{HK:free}
{\sc P.~Hackney {\rm and }J.~Kock}.
\newblock {\em Free decomposition spaces}.
\newblock Preprint. \arxiv{2210.11192}.

    \bibitem{Hirschhorn}
{\sc P.~Hirschhorn}.
\newblock {\em Model Categories and Their Localizations}, vol.~99 of Mathematical Surveys and Monographs.
\newblock American Mathematical Society, 2003.

    \bibitem{Hovey}
{\sc M.~Hovey}.
\newblock {\em Model Categories}, vol.~63 of Mathematical Surveys and Monographs.
\newblock American Mathematical Society, 1999.

    \bibitem{Joyal:published}
{\sc A.~Joyal}.
\newblock {\em Quasi-categories and {K}an complexes}.
\newblock Journal of Pure and Applied Algebra {\bf 175} (2002), pp.~207--222.

    \bibitem{Joyal:notes}
{\sc A.~Joyal}.
\newblock {\em Notes on quasi-categories}.
\newblock \url{https://www.math.uchicago.edu/~may/IMA/Joyal.pdf}

    \bibitem{Joyal:theory}
{\sc A.~Joyal}.
\newblock {\em The theory of quasi-categories
and its applications}. \\
\url{https://mat.uab.cat/~kock/crm/hocat/advanced-course/Quadern45-2.pdf}

    \bibitem{JT}
{\sc A.~Joyal {\rm and }M.~Tierney}.
\newblock {\em Quasi-categories vs {S}egal spaces}.
\newblock Categories in algebra, geometry and
mathematical physics, Contemp.~Math.~{\bf 431} (2007), Amer.~Math.~Soc., pp.~277–326.

    \bibitem{Lurie:HTT}
{\sc J.~Lurie}.
\newblock {\em Higher Topos Theory}, vol.~170 of Annals of Mathematics Studies.
\newblock Princeton University Press, Princeton, NJ, 2009.

    \bibitem{Poguntke}
{\sc T.~Poguntke}.
\newblock {\em Higher {S}egal structures in algebraic {K}-theory}.
\newblock Preprint. \arxiv{1709.06510}.

    \bibitem{Quillen}
{\sc D.~Quillen}.
\newblock {\em Homotopical algebra}, Lecture Notes in Mathematics, no.~43, {\bf 43}.
\newblock Springer-Verlag, Berlin, New York, 1967.


    \bibitem{Rasekh:Yoneda}
{\sc N.~Rasekh}.
\newblock {\em Yoneda Lemma for Simplicial spaces}.
\newblock Preprint. \arxiv{1711.03160}

    \bibitem{Rezk:CSS}
{\sc C.~Rezk}.
\newblock {\em A model for the homotopy theory of homotopy theory}.
\newblock Trans.~Amer.~Math.~Soc.~{\bf 353} (2001), no.~3, pp.~973-–1007.

    \bibitem{Rezk}
{\sc C.~Rezk}.
\newblock {\em Introduction to quasicategories}.\\
\newblock \url{https://faculty.math.illinois.edu/~rezk/quasicats.pdf}

    \bibitem{Stern}
{\sc W.~Stern}.
\newblock {\em 2-{S}egal objects and algebras in spans}.
\newblock J. Homotopy Relat. Struct. {\bf 16} (2021), pp.~297–361.

    \bibitem{Walde:operads}
{\sc T.~Walde}.
\newblock {\em 2-{S}egal spaces as invertible infinity-operads}.
\newblock Algebr.~Geom.~Topol.~{\bf 21} (2021), pp.~211--246.

    \bibitem{Young}
{\sc M.~Young}.
\newblock {\em Relative 2–{S}egal spaces}.
\newblock Alg.~Geom.~Topol. {\bf 18} (2018), pp.~975–1039.

\end{thebibliography}
\end{document}